\documentclass[1 leqno,11pt,psfig]{amsart}
\usepackage{amssymb, amsmath,amsmath,latexsym,amssymb,amsfonts,amsbsy, amsthm,mathtools,graphicx,CJKutf8,CJKnumb,CJKulem,color}
\usepackage{graphicx,epsfig}
\usepackage{graphicx}
\usepackage{amssymb}
\usepackage{graphicx}
\usepackage{graphicx,epsfig}
\usepackage{amsmath}
\usepackage{mathrsfs}
\usepackage{amssymb}

\usepackage{amssymb,mathrsfs,amsfonts,amsmath,amsbsy,tikz}\usepackage{fontenc}\usepackage{textcomp}

\numberwithin{equation}{section}

\allowdisplaybreaks

\newcommand{\beq}{\begin{equation}}
\newcommand{\eeq}{\end{equation}}
\newcommand{\ben}{\begin{eqnarray}}
\newcommand{\een}{\end{eqnarray}}
\newcommand{\beno}{\begin{eqnarray*}}
\newcommand{\eeno}{\end{eqnarray*}}

\newtheorem{theorem}{Theorem}[section]

\newtheorem{lemma}[theorem]{Lemma}
\newtheorem{proposition}[theorem]{Proposition}
\newtheorem{corol}[theorem]{Corollary}
\newtheorem{example}[theorem]{Example}
\newtheorem{remark}[theorem]{Remark}

\begin{document}
\begin{CJK*}{UTF8}{gkai}
\title[Dynamics on the space of probability measures]{Dynamics of continuous maps  induced on the space of
probability measures}

\author{Hua Shao}
\address{Department of Mathematics, Nanjing University of Aeronautics and Astronautics,
 211106, Nanjing, P. R. China}
\email{huashao@nuaa.edu.cn}

\author{Hao Zhu}
\address{Chern Institute of Mathematics, Nankai University, 300071, Tianjin, P. R. China}
\email{haozhu@nankai.edu.cn}

\author{Guanrong Chen}
\address{Department of Electrical Engineering, City University of Hong Kong, Hong Kong SAR, P.~R. China}
\email{eegchen@cityu.edu.hk}

\date{\today}

\maketitle

\begin{abstract}
For   a continuous self-map $f$  on a compact interval $I$ and  the induced map $\hat f$ on the space $\mathcal{M}(I)$ of
probability measures,
we obtain a sharp condition  to guarantee that $(I,f)$ is transitive if and only if $(\mathcal{M}(I),\hat f)$ is transitive.
We also show that the sensitivity of $(I,f)$ is equivalent to that of $(\mathcal{M}(I),\hat f)$.
We prove that $(\mathcal{M}(I),\hat f)$ must have infinite topological entropy for any transitive system  $(I,f)$, while there exists a transitive non-autonomous system $(I,f_{0,\infty})$ such that $(\mathcal{M}(I),\hat f_{0,\infty})$ has zero topological entropy, where $f_{0,\infty}=\{f_n\}_{n=0}^\infty$ is a sequence of continuous self-maps on $I$. For a continuous self-map $f$ on a general compact metric space $X$, we show that
chain transitivity of  $(X, f)$ implies chain mixing of  $(\mathcal{M}(X),\hat f)$, and we  provide two counterexamples to demonstrate that the converse is not true.
We confirm that shadowing of $(X,f)$ is not inherited by $(\mathcal{M}(X),\hat f)$ in general.
For a non-autonomous system $(X,f_{0,\infty})$, we prove that if $(\mathcal{M}(X),\hat{f}_{0,\infty})$ is
weak mixing of order $n$, then  so is $(X,f_{0,\infty})$ for any $n\geq2$; while  there exists  $(X,f_{0,\infty})$ such that it is
weak mixing of order $2$ but
 $(\mathcal{M}(X),\hat{f}_{0,\infty})$ is not.
We then prove that Li-Yorke chaos (resp., distributional chaos) of $(X,f_{0,\infty})$ carries over to $(\mathcal{M}(X),\hat f_{0,\infty})$, and
 give an example to show that $(X,f)$ and $(\mathcal{M}(X),\hat f)$ may have no Li-Yorke pair simultaneously.
We  also prove that if $f_n$ is surjective for all $n\geq 0$, then chain mixing of $(\mathcal{M}(X),\hat f_{0,\infty})$ always holds true, and
 shadowing of $(\mathcal{M}(X),\hat f_{0,\infty})$ implies  mixing of  $(X, f_{0,\infty})$.
\end{abstract}



\section{Introduction}
Let $X$ be a compact metric space with metric $d$.
A  non-autonomous (topological) dynamical system is a pair $(X,f_{0,\infty})$, where $f_{0,\infty}=\{f_n\}_{n=0}^{\infty}$
is a sequence of continuous self-maps on $X$. For any $x_0\in X$, the (positive) orbit $\{x_n\}_{n=0}^{\infty}$ of $(X,f_{0,\infty})$ starting from $x_0$ is defined by $x_n=f_0^n(x_0)$, where $f_0^n=f_{n-1}\circ\cdots\circ f_{0}$, $n\geq 1$. Note that if $f_n=f$ for all $n\geq0$,
then $(X, f_{0,\infty})$ becomes the autonomous  dynamical system  $(X,f)$.

Let $\mathcal{B}(X)$ be the $\sigma$-algebra of all Borel subsets of $X$ and $\mathcal{M}(X)$ be the space of all
Borel probability measures on $X$. The Prohorov metric $\mathcal{P}_{d}$ on $\mathcal{M}(X)$ is defined by
\begin{align*}
\mathcal{P}_{d}(\mu,\nu)=\inf\{\epsilon>0: \mu(A)\leq\nu(A^{\epsilon})+\epsilon\;\;{\rm and}\; \nu(A)\leq\mu(A^{\epsilon})+\epsilon,
\;A\in\mathcal{B}(X)\}
\end{align*}
for $\mu,\nu\in\mathcal{M}(X)$, where $A^{\epsilon}=\{x\in X: d(x,A)<\epsilon\}$. It was proved in \cite{Strassen} that
\begin{align}\label{measure-metric-def}
\mathcal{P}_{d}(\mu,\nu)=\inf\{\epsilon>0: \mu(A)\leq\nu(A^{\epsilon})+\epsilon,\;A\in\mathcal{B}(X)\}.
\end{align}
The topology induced by the metric $\mathcal{P}_{d}$ coincides with the weak$^{*}$-topology for measures.
$(X, f_{0,\infty})$ induces $(\mathcal{M}(X),\hat{f}_{0,\infty})$, where $\hat{f}_{0,\infty}=\{\hat{f}_n\}_{n=0}^{\infty}$,
with $\hat{f}_n: \mathcal{M}(X)\to \mathcal{M}(X)$ defined by
\begin{align}\label{1.1}
\hat{f}_n(\mu)(A)=\mu\big(f_{n}^{-1}(A)\big),\;\mu\in \mathcal{M}(X),\;A\in\mathcal{B}(X).
\end{align}
Denote $\hat{f}_{0}^{n}=\hat{f}_{n-1}\circ\cdots\circ\hat{f}_{0}$ and $f_{0}^{-n}=(f_{0}^{n})^{-1}$. Then (\ref{1.1}) ensures that
\begin{align*}
\hat{f}_{0}^{n}(\mu)(A)=\mu\big(f_{0}^{-n}(A)\big),\;\mu\in\mathcal{M}(X),\;A\in\mathcal{B}(X),\;n\geq1.
\end{align*}
$(\mathcal{M}(X),\mathcal{P}_{d})$ is a compact metric space and $\hat{f}_n$ is continuous on $\mathcal{M}(X)$ for
$n\geq0$ (see \cite{Bauer,Parthasarathy} for more details).

$(\mathcal{M}(X),\hat{f}_{0,\infty})$ is a dynamical system with deterministic dynamics and  stochastic configurations in  statistical mechanics.
It is natural to consider what kind of  dynamics of  $(X,f_{0,\infty})$ can be inherited by $(\mathcal{M}(X),\hat{f}_{0,\infty})$, and
conversely how the dynamical behaviors of $(\mathcal{M}(X),\hat{f}_{0,\infty})$ affect those of $(X,f_{0,\infty})$.
 In 1975, Bauer and Sigmund  \cite{Bauer} introduced the induced dynamical  system $(\mathcal{M}(X),\hat{f})$ of an autonomous dynamical system $(X,f)$, and studied
  systematically what  topological properties of $(X,f)$  can carry over to $(\mathcal{M}(X),\hat{f})$.
Since then, the study on the interrelations of dynamics between $(X,f)$ and $(\mathcal{M}(X),\hat{f})$
has been developed rapidly, and we refer the readers to \cite{Bernardes,Glasner,Komuro,Li13,Li17,Wu}
and references therein. It is worthy to mention that
Glasner and Weiss  \cite{Glasner} proved that for a minimal dynamical system $(X, f)$,  it has  zero topological entropy
if and only if $(\mathcal{M}(X),\hat{f})$ has  zero topological entropy, which demonstrates a big difference of dynamics between $(\mathcal{K}(X),\bar{f})$ and $(\mathcal{M}(X),\hat{f})$. Here, $\bar{f}$ is the induced map on the hyperspace $\mathcal{K}(X)$, and the readers are referred to \cite{FG16,Guirao,Liu,Li13,Peris,Roman03} for the dynamics of $(\mathcal{K}(X),\bar{f})$. For the generic  homeomorphism $f$ on the Cantor space $\{0,1\}^{\mathbf{N}}$,
 Bernardes and Vermersch  proved that $(\mathcal{M}(\{0,1\}^{\mathbf{N}}),\hat f)$ has no Li-Yorke pair in \cite{Bernardes}.
  Li {\it et al.}  \cite{Li17} showed that  multi-$\mathcal{F}$-sensitivity
of $(X, f)$ is equivalent to that of $(\mathcal{M}(X),\hat{f})$. They also proved that  Li-Yorke sensitivity of $(\mathcal{M}(X),\hat{f})$ implies that of $(X,f)$, but the converse is not true in general. Li {\it et al.}  showed that $(\mathcal{M}(X),\hat{f})$
is a $P$-system if and only if $(X,f)$ is a weakly mixing almost-HY-system in \cite{Li13}; and  Wu  \cite{Wu} proved that the exactness of $(X,f)$ is equivalent to that of
$(\mathcal{M}(X),\hat{f})$.



In this paper, we obtain some  new results on the interrelations of    dynamics between a dynamical system and its induced system on the space of probability measures.
{\bf Firstly},   our  results for
 the interval  dynamical  system  $(I,f)$ are provided in (1)--(4).

(1) A sharp condition is given to ensure the  transitivity of $(I,f)$ be  equivalent to that of  $(\mathcal{M}(I),\hat{f})$, see Theorem \ref{sharp condition transitivity on interval}.


(2) 
 $(I,f)$ is totally transitive if and only if $(\mathcal{M}(I), \hat f)$ is totally transitive, see Theorem \ref{total transitivity}.

(3) $(\mathcal{M}(I),\hat{f})$ has  infinite topological entropy for any transitive system $(I,f)$. However,
there exists a  transitive non-autonomous system $(I,f_{0,\infty})$ such that $(\mathcal{M}(I),\hat{f}_{0,\infty})$ has zero topological entropy,
see Theorem \ref{transitivity-entropy}.

(4) $(I,f)$ is sensitive if and only if $(\mathcal{M}(I),\hat{f})$ is sensitive, see Theorem \ref{sensitive equivalent interval}.

{\bf Secondly,} our results  for the autonomous  dynamical  system $(X, f)$ are the following (5)--(6).

(5) Chain transitivity of $(X, f)$ implies chain mixing of $(\mathcal{M}(X),\hat{f})$, see Corollary \ref{adschain}. Two counterexamples are given to show that the converse is not true in general, see Examples \ref{surjective} and \ref{homeomorphism}.

(6) There exists $(X, f)$ such that it has  shadowing  but $(\mathcal{M}(X),\hat{f})$ has no shadowing, see Theorem \ref{discrete-example}.

(7) A simple  $(X,f)$  is given  such that $(X,f)$ and $(\mathcal{M}(X),\hat f)$  have no Li-Yorke pair simultaneously, but 
$(\mathcal{K}(X),\bar f)$ is  distributionally chaotic, see Example \ref{one-point compactification of integers no Li-Yorke}.

It is noted that
 the majority of complex  systems in the fields of biology, physics and engineering are driven by  a sequence of different  functions, and thus
the study on non-autonomous dynamical systems is of significant importance in applications (see \cite{BO12,Kolyada96,ShaoChaos,Shi09} and references therein).
{\bf Finally,} our results for the  non-autonomous  dynamical  system $(X, f_{0,\infty})$ are given in (7)--(10).

(8) There exists $(X, f_{0,\infty})$ such that it is  weak mixing of order $2$
 but $(\mathcal{M}(X),\hat{f}_{0,\infty})$ is not, while weak mixing of order $n$ of $(\mathcal{M}(X),\hat{f}_{0,\infty})$ implies  that of $(X,f_{0,\infty})$ for any $n\geq2$, see Theorem \ref{weak mixing order 2} and Proposition \ref{weak mixing some order all orders}.

(9) If $f_n$ is surjective for all $n\geq0$, then $(\mathcal{M}(X),\hat{f}_{0,\infty})$ is always chain mixing, and  shadowing of $(\mathcal{M}(X), \hat f_{0,\infty})$ ensures  mixing of  $(X,  f_{0,\infty})$, see Theorem \ref{chain mixing}  and Corollary \ref{all}.

(10) $(X, f_{0,\infty})$ is  mixing (resp., mild mixing,  exact) if and only if $(\mathcal{M}(X),$
$\hat{f}_{0,\infty})$ is so, see Theorems \ref{mixing}--\ref{exactness}.

(11) $(X, f_{0,\infty})$ is  cofinitely sensitive (resp., multi-sensitive) if and only if $(\mathcal{M}(X),$
$\hat{f}_{0,\infty})$ is so, see Theorems \ref{cofinite}--\ref{multi sensitive}.


(12) 
 $(X, f_{0,\infty})$ is  equi-conjugate to $(Y, g_{0,\infty})$ if and only if $(\mathcal{M}(X),\hat{f}_{0,\infty})$ is  equi-conjugate to $(\mathcal{M}(Y),$ $\hat{g}_{0,\infty})$, see Theorem \ref{equi conjugacy}.

Here, we compare the existing results in the literatures with  our new results in this paper.
Bauer and Sigmund  \cite{Bauer} proved that transitivity of $(\mathcal{M}(X),\hat{f})$ implies that of $(X, f)$, but the converse is not true in general. In Theorem \ref{sharp condition transitivity on interval}, we obtain  a sharp criterion for the converse to be valid on an interval.
 They also proved that weak mixing of order $2$ is equivalent between  $(X,f)$ and $(\mathcal{M}(X),\hat{f})$.  In  Theorem \ref{weak mixing order 2} and Proposition \ref{weak mixing some order all orders}, we show that  weak mixing of order $2$ of $(X,f_{0,\infty})$ can not carry over to
 $(\mathcal{M}(X),\hat{f}_{0,\infty})$, while weak mixing of order $n$ of $(\mathcal{M}(X),\hat{f}_{0,\infty})$ is inherited by  $(X,f_{0,\infty})$  for $n\geq2$.
 Fernandez {\it et al.}  \cite{FGPR15} gave a counterexample to show that  chain transitivity of $(X, f)$ does not imply that of $(\mathcal{K}(X),\bar{f})$ in general. This is different for $(\mathcal{M}(X),\hat{f})$, for which we show in Corollary \ref{adschain} that if $(X, f)$ is chain transitive, then so is $(\mathcal{M}(X),\hat{f})$. Li {\it et al.} \cite{Li17} showed that the sensitivity of $(X,f)$ is not inherited by  $(\mathcal{M}(X),\hat{f})$ in general. On an interval, we prove in Theorem \ref{sensitive equivalent interval} that  $(I,f)$ is sensitive if and only if  $(\mathcal{M}(I),\hat{f})$ is sensitive.
Bernardes and Vermersch \cite{Bernardes} showed that if $f$ is  homeomorphic, then $(\mathcal{M}(X),\hat{f})$ is chain mixing, and
 weak shadowing of  $(\mathcal{M}(X), \hat f)$ implies transitivity of $(X, f)$. We furthermore prove  in Theorem \ref{chain mixing} and Theorem \ref{mixingshadowing}  that if $f_n$ is surjective for $n\geq0$, then  $(\mathcal{M}(X),\hat{f}_{0,\infty})$ is chain mixing, and shadowing of  $(\mathcal{M}(X), \hat f_{0,\infty})$ implies  mixing of $(X, f_{0,\infty})$.
Guirao {\it et al.}  provided an example $(X,f)$ which has no Li-Yorke pair such that $(\mathcal{K}(X),\bar{f})$ is Li-Yorke chaotic in \cite{Guirao}. For this $(X, f)$, in contrast to $(\mathcal{K}(X),\bar{f})$, we prove that  $(\mathcal{M}(X),\hat{f})$ has no Li-Yorke pair in Example \ref{one-point compactification of integers no Li-Yorke}.
Fern\'andez and  Good  showed that shadowing of $(X, f)$ and $(\mathcal{K}(X),\bar{f})$ are equivalent in \cite{FG16}. However,  we
obtain in Theorem \ref{discrete-example} that shadowing of $(X, f)$ does not imply  that of $(\mathcal{M}(X),\hat{f})$ in general.

The rest of the paper is organized as follows. Sections 2-5 investigates  the interrelations of different  dynamics  between a dynamical system and its induced system on the space of probability measures.  In Section 2, transitivity, mixing and exactness  are considered.   Chain mixing, chain transitivity, shadowing and specification
 are studied in Section 3.
Section 4 deals with various kinds of sensitivity. In Section 5,   topological sequence  entropy,  conjugacy,  Li-Yorke chaos and distributional chaos
are investigated.

\section{Transitivity and mixing}
In this section,  the interrelations of  transitivity and mixing between
$(X,f_{0,\infty})$ and $(\mathcal{M}(X),$ $\hat f_{0,\infty})$ are studied.

Let $\mathbf{N}$ and $\mathbf{Z^{+}}$ denote the set of all nonnegative integers and the set of all positive
integers, respectively, and let $B_{d}(x,\varepsilon)$ and $\bar{B}_{d}(x,\varepsilon)$ denote the open and closed
balls of radius $\varepsilon>0$ centered at $x\in X$, respectively.

First, recall some concepts and properties. $(X,f_{0,\infty})$ is (topologically) transitive  if $N(U,V)\neq\emptyset$ for any two nonempty open subsets $U$
and $V$ of $X$, where $N(U,V):=\{n\geq 1:f_{0}^{n}(U)\cap V\neq\emptyset\}$; it is mild mixing if
$(X\times Y, f_{0,\infty}\times g_{0,\infty})$ is transitive for any transitive
system $(Y, g_{0,\infty})$; it is (topologically) mixing if there exists $N_0\in \mathbf{Z^{+}}$ such that
$N(U,V)\supset[N_0,+\infty)\cap\mathbf{Z^{+}}$ for any two nonempty open subsets $U$
and $V$ of $X$; it is (topologically) exact if there exists $N_1\in \mathbf{Z^{+}}$
such that $f^{n}_{0}(U)=X$  for any  nonempty open subset $U$ of $X$ and for all $n\geq N_1$; it is weakly mixing of order $n$ if $\bigcap_{i=1}^{n}N(U_i,V_i)
\neq\emptyset$ for any nonempty open subsets $U_1,\cdots,U_n$ and $V_1,\cdots,V_n$; and it is weakly mixing of all orders
if $(X,f_{0,\infty})$ is weakly mixing of order $n$ for every $n\geq2$.
It is evident that exactness\;$\Rightarrow$\;mixing\;$\Rightarrow$\;mild mixing\;$\Rightarrow$\;weakly mixing of all orders\;$\Rightarrow$\;weakly mixing of order $n$ ($n\geq2$)\;$\Rightarrow$\;transitivity for $(X,f_{0,\infty})$.

Then some preliminaries on the probability measures are given.
Let $x\in X$ and $\delta_{x}\in\mathcal{M}(X)$ be the Dirac point measure of $x$ defined by
\begin{align*}
		\delta_{x}(A)=\left\{\begin{array}{ll}
		1,& x\in A,\\
		0, & x\notin A,\\
\end{array}\right.
\end{align*}
for $A\in \mathcal{B}(X)$. Note that $\hat{f}_{n}(\delta_{x})=\delta_{f_{n}(x)}$ for $x,y\in X$ and $n\geq0$.
Some basic results summarized below can be found in, e.g. \cite{Bauer,Bernardes,Komuro}.

\begin{lemma}\label{basic result}
{\rm(i)} $\mathcal{P}_{d}(\delta_x,\delta_y)=\min\{d(x,y),1\}$ for
$x,y\in X$.

{\rm(ii)}  Let $\mathcal{M}_{n}(X)\triangleq\left\{\frac{1}{n}\sum_{i=1}^{n}\delta_{x_i}: x_1,\cdots,x_n\in X\right\}$ for
$n\geq1$. Then $\mathcal{M}_{n}(X)$ is closed in $\mathcal{M}(X)$ for any $n\geq1$, and $\mathcal{M}_{\infty}(X)\triangleq\cup_{n=1}^{\infty}\mathcal{M}_{n}(X)$ is dense in $\mathcal{M}(X)$.

{\rm(iii)} Let $\mu_i\in\mathcal{M}(X)$, $i=1,\cdots, n$, $f: X\to X$, $\alpha_i\in\mathbb{R}$,  and $\sum_{i=1}^n \alpha_i=1$. Then
$\hat f(\sum_{i=1}^n\alpha_i\mu_i)=\sum_{i=1}^n \alpha_i\hat f(\mu_i).$

{\rm(iv)} Let $\mu,\nu\in\mathcal{M}(X)$ and $0\leq \alpha\leq \beta\leq 1$. Then $\mathcal{P}_d(\alpha\mu+(1-\alpha)\nu,\beta\mu+(1-\beta)\nu)\leq\beta-\alpha.$
\end{lemma}

Let $X^{n}\triangleq\underbrace{X\times\cdots\times X}_{n}$, and
define a map $\varphi_n: X^{n}\to \mathcal{M}_{n}(X)$ by
\begin{align}\label{2.3}
\varphi_n(x)=\frac{1}{n}\sum_{i=1}^{n}\delta_{x_i},\;x=(x_1,\cdots,x_n)\in X^{n},\;n\geq1.
\end{align}
Define the metric $d_n$ in $X^{n}$ by $d_n(x,y)=\max_{1\leq i\leq n}d(x_i,y_i)$
for $x=(x_1,\cdots,x_n)$ and $y=(y_1,\cdots,y_n)\in X^{n}$.
It is easy to verify that $\varphi_n$ is continuous in $X^{n}$ for $n\geq1$.


\subsection{Transitivity}
First, a preliminary result is established.

\begin{proposition}\label{transitive}
If $(\mathcal{M}(X),\hat f_{0,\infty})$ is transitive, then so is $(X,f_{0,\infty})$.
\end{proposition}

\begin{proof}
 Fix any two nonempty open subsets $U$ and $V$ in $X$. Denote
\begin{align*}
\mathcal{U}:=\{\mu\in\mathcal{M}(X): \mu(U)>1/2\},\;\mathcal{V}:=\{\nu\in\mathcal{M}(X): \nu(V)>1/2\}.
\end{align*}
It is easy to verify that $\mathcal{U}$ and $\mathcal{V}$ are nonempty open subsets of $\mathcal{M}(X)$.
Since $(\mathcal{M}(X),\hat f_{0,\infty})$ is transitive, there exist $N_0\geq 1$ and $\mu\in\mathcal{U}$ such that
 $\hat{f}_{0}^{N_0}(\mu)\in\mathcal{V}$.
Thus, $\mu(U)>1/2$ and $\hat{f}_{0}^{N_0}(\mu)(V)=\mu\big(f_{0}^{-N_0}(V)\big)>1/2$,
which gives
$f_{0}^{N_0}(U)\cap V\neq\emptyset.$
Hence, $(X,f_{0,\infty})$ is transitive.
\end{proof}

In \cite{Bauer}, it was shown that the irrational rotation mapping $R: \mathbb{S}^1\to \mathbb{S}^1$ is
transitive but $\hat R: \mathcal{M}(\mathbb{S}^1)\to \mathcal{M}(\mathbb{S}^1)$ is not, where $\mathbb{S}^1$ is the unit circle.
Thus, the converse of Proposition \ref{transitive} may not be true even for a single map. Nevertheless, we get
a sharp condition to guarantee the validness of the converse  for an interval map $f$.

\begin{theorem}\label{sharp condition transitivity on interval}
{\rm(i)}
Assume that $(I,f)$ has a periodic point of odd period different from $1$. Then $(I,f)$ is transitive if and only if $(\mathcal{M}(I),\hat f)$ is transitive.

{\rm(ii)} There exists  $(I,f)$ such that

\qquad{\rm($1$)} it has no periodic points of odd period different from $1$;

\qquad{\rm($2$)} it is transitive but  $(\mathcal{M}(I),\hat f)$ is not transitive.
\end{theorem}
\begin{proof}
(i)
By Proposition \ref{transitive}, the transitivity of $(\mathcal{M}(I),\hat f)$ implies that of $(I,f)$. Conversely, transitivity of $(I,f)$ is equivalent to  mixing of $(I,f)$ by Theorem 2.20 in \cite{R17}.  Theorem \ref{mixing} below shows that  mixing of $(I,f)$ is equivalent to that of $(\mathcal{M}(I),\hat f)$. This means that $(\mathcal{M}(I),\hat f)$ is transitive.

 (ii)
Let $I=[-1,1]$ and
\begin{align*}
		f(x)=\left\{\begin{array}{ll}
		2x+2,& x\in [-1,-{1\over2}],\\
		-2x, & x\in [-{1\over2},0],\\
        -x, & x\in [0,1],\\
		\end{array}\right.
	\end{align*}
which is illustrated in Figure 1. Then $f([-1,0])=[0,1]$ and $f([0,1])=[-1,0]$.
By Theorem 2.20 and Example 2.21 in \cite{R17}, $(I,f)$ is transitive and has no periodic points of odd period different from $1$.
 \begin{center}
 \begin{tikzpicture}[scale=0.7]
 \draw [->](-3, 0)--(3, 0)node[right]{${ x}$};
 \draw [->](0,-3)--(0,3) node[above]{${ f}$};
   \draw (-2, 0).. controls (-1, 2) and (-1, 2)..(-1, 2);
    \draw (-1, 2).. controls (0, 0) and (0, 0)..(0,0);
   \draw (0, 0).. controls (2, -2) and (2, -2)..(2, -2);
    \path (-1, 0)  edge [-,dotted](-1, 2) [line width=0.4pt];
    \path (-1, 2)  edge [-,dotted](0,2) [line width=0.4pt];
    \path (0, -2)  edge [-,dotted](2, -2) [line width=0.4pt];
    \path (2, 0)  edge [-,dotted](2,-2) [line width=0.4pt];
    \node (a) at (2,0.3) {\small$1$};
    \node (a) at (-0.5,-2) {\small$-1$};
     \node (a) at (-1,-0.4) {\small$-{1\over2}$};
      \node (a) at (0.2,2) {\small$1$};
       \node (a) at (-2.1,-0.3) {\small$-1$};
 \end{tikzpicture}
\end{center}
 \begin{center}\vspace{-0.2cm}
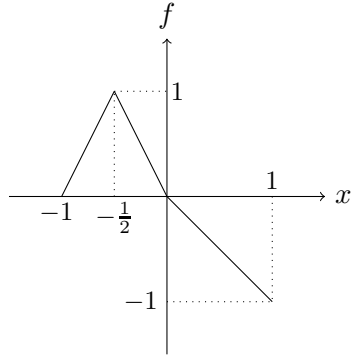
   \small { Figure 1. The piecewise-linear map}
  \end{center}

Then it will be shown that $(\mathcal{M}(I),\hat{f})$ is not transitive.
Let $\mu_1=\delta_{-{1\over2}}$, $\mu_2={1\over2}(\delta_{-{1\over2}}+\delta_{1\over2})$ and $\varepsilon_0\in(0,{1\over4})$.
First, we claim that for any $\nu\in B_{\mathcal{P}_d}(\mu_1,\varepsilon_0)$ and for any Borel subset $K\subset [0,1]$, we have
\begin{align}\label{nukleqv0}
\nu(K)\leq \varepsilon_0.
\end{align}
 Suppose that this is not true. Then  there exist $\nu_0\in B_{\mathcal{P}_d}(\mu_1,\varepsilon_0)$ and a Borel subset $K_0\subset [0,1]$ such that $\nu_0(K_0)>\varepsilon_0$.
Since $\mu_1(K_0^{\varepsilon_0})=\delta_{-{1\over 2}}(K_0^{\varepsilon_0})=0$, we have $\nu_0(K_0)>\mu_1(K_0^{\varepsilon_0})+\varepsilon_0$, which yields
$\mathcal{P}_d(\nu_0,\mu_1)>\varepsilon_0$. However, this contradicts the fact that $\nu_0\in B_{\mathcal{P}_d}(\mu_1,\varepsilon_0)$.


Let $\varepsilon\in(0,{1\over2}-\varepsilon_0)$. Now, it can be shown that $\hat f^{n}(B_{\mathcal{P}_d}(\mu_1,\varepsilon_0))\cap  B_{\mathcal{P}_d}(\mu_2,\varepsilon)=\emptyset$ for all  $n\geq0$. Let $\nu\in B_{\mathcal{P}_d}(\mu_1,\varepsilon_0)$.
First, consider $n$ to be even. Take $A=\{{1\over2}\}$. Since $A^{\varepsilon}\subset [0,1]$ and $f^{-n}([0,1])=[0,1]$ for even $n$, it follows from (\ref{nukleqv0}) that  $\nu(f^{-n}(A^{\varepsilon}))\leq \varepsilon_0$. Thus, $\mu_2(A)={1\over 2}> \nu(f^{-n}(A^{\varepsilon}))+\varepsilon=\hat f^n(\nu)(A^{\varepsilon})+\varepsilon$, which gives
$\mathcal{P}_d(\hat f^{n}(\nu),\mu_2)>\varepsilon$ for even $n$. Next, consider $n$ to be odd. Set $B=\{-{1\over2}\}$.
Note that $B^{\varepsilon}\subset [-1,0]$ and $f^{-n}([-1,0])=[0,1]$ for odd $n$. Thus, $\nu(f^{-n}(B^{\varepsilon}))\leq \varepsilon_0$.
So, $\mu_2(B)={1\over 2}> \nu(f^{-n}(B^{\varepsilon}))+\varepsilon=\hat f^n(\nu)(B^{\varepsilon})+\varepsilon$, which implies
$\mathcal{P}_d(\hat f^{n}(\nu),\mu_2)>\varepsilon$ for odd $n$. This proves that $\hat f^{n}(B_{\mathcal{P}_d}(\mu_1,\varepsilon_0))\cap  B_{\mathcal{P}_d}(\mu_2,\varepsilon)=\emptyset$ for all  $n\geq0$.
\end{proof}

The following results can be obtained for a single map $f$ on $I$ unconditionally. Recall that
$(I,f)$ is  totally transitive if $(I,f^{n})$ is transitive for $n\geq1$.

\begin{theorem}\label{total transitivity}
{\rm(i)} $(I,f^2)$ is transitive if and only if $(\mathcal{M}(I), \hat f^2)$ is  transitive.

{\rm(ii)} $(I,f)$ is totally transitive if and only if $(\mathcal{M}(I), \hat f)$ is totally transitive.
\end{theorem}
\begin{proof}
By Theorem 2.20 in \cite{R17},  transitivity of $(I,f^2)$ (resp., total transitivity of $(I,f)$) is equivalent
to  mixing of $(I,f)$. It follows from Theorem \ref{mixing} below that   mixing of $(I,f)$ is equivalent to that of
$(\mathcal{M}(I), \hat f)$, which yields transitivity of $(\mathcal{M}(I), \hat f^2)$ (resp., total transitivity
of $(\mathcal{M}(I), \hat f)$). Conversely, it follows from Proposition \ref{transitive} that transitivity of
$(\mathcal{M}(I), \hat f^2)$ (resp., total transitivity of $(\mathcal{M}(I), \hat f)$) ensures transitivity
of $(I,f^2)$ (resp., total transitivity of $(I,f)$).
\end{proof}

Dynamics of non-autonomous systems is much richer than that of autonomous systems in general. Finally, it will be  proved that $(\mathcal{M}(I), \hat f)$ has infinite topological entropy for any transitive autonomous system $(I,f)$, while there exists a transitive non-autonomous system  $(I,f_{0,\infty})$ such that $(\mathcal{M}(I), \hat f)$ has zero topological entropy. To proceed, we
 recall the definition of topological (sequence) entropy of $(X,f_{0,\infty})$ introduced in \cite{Kolyada96,Sotola}, which will be discussed in detail in Section 5. Suppose that $A=\{a_i\}_{i=1}^{\infty}\subset\mathbf{Z^{+}}$
is an increasing sequence, $n\geq1$ and $\epsilon>0$. A subset $E\subset X$ is  called $(n,\epsilon,A)$-separated if, for any $x\neq y\in E$,
there exists $0\leq j\leq n-1$ such that $d(f_{0}^{a_j}(x),f_{0}^{a_j}(y))>\epsilon$. Let $\Lambda\subset X$ and $s_{n}(\epsilon,A,f_{0,\infty},\Lambda)$
be the maximal cardinality of an $(n,\epsilon,A)$-separated set in $\Lambda$. Denote
\begin{align}\label{topological sequence entropy}
h_{A}(f_{0,\infty},\Lambda):=\lim_{\epsilon\to0}\limsup_{n\to\infty}\frac{1}{a_n}\log s_{n}(\epsilon,A,f_{0,\infty},\Lambda).
\end{align}
If $\Lambda=X$, then $h_A(f_{0,\infty}):=h_{A}(f_{0,\infty},X)$ is called the topological sequence entropy
of $(X,f_{0,\infty})$ with respect to the sequence $A$. Further, if $A=\mathbf{Z^{+}}$, then
$h(f_{0,\infty}):=h_{\mathbf{Z^{+}}}(f_{0,\infty})$ is called the topological entropy of $(X,f_{0,\infty})$.
If $f_n=f$ for all $n\geq0$, then $h(f_{0,\infty})$ is briefly denoted as $h(f)$.

\begin{lemma}\label{converge uniform}
Let $f_n:X\to X$ be a map for  $n\geq 0$. Then $f_n$ converges uniformly to $f$ on $X$ if and only if $\hat f_n$ converges uniformly to $\hat f$  on  $\mathcal{M}(X)$.
\end{lemma}

\begin{proof}
If $f_n$ converges uniformly to $f$ on $X$, then  for any $\varepsilon>0$, there exists $N>0$  such that  $d(f_n(x),f(x))<\varepsilon$ for  $n\geq N$ and  $x\in X$.  Thus, $f_n^{-1}(A)\subset f^{-1}(A^\varepsilon)$ for  $n\geq N$ and  $A\in\mathcal{B}(X)$.  Let     $\mu\in\mathcal{M}(X)$. Then $f_n^{-1}(A)\subset f^{-1}(A^\varepsilon)$ implies $\hat f_n(\mu)(A)=\mu(f_n^{-1}(A))\leq \mu(f^{-1}(A^\varepsilon))=\hat f(\mu)(A^{\varepsilon})\leq \hat f(\mu)(A^{\varepsilon})+\varepsilon$ for  $A\in \mathcal{B}(X)$ and  $n\geq N$, which yields $\mathcal{P}_d(\hat f_n(\mu),\hat f(\mu))<\varepsilon$.
Thus,  $\hat f_n$ converges uniformly to $\hat f$.

Conversely, for any $\varepsilon\in (0,1)$, there exists $N>0$ such that $\mathcal{P}_d(\hat f_n(\mu),\hat f(\mu))<\varepsilon$ for  $n\geq N$ and $\mu\in\mathcal{M}(X)$.  By Lemma \ref{basic result} (i),  $d(f_n(x),f(x))=\mathcal{P}_d(\delta_{f_n(x)},\delta_{f(x)})=\mathcal{P}_d(\hat f_n(\delta_{x}),\hat f(\delta_{x}))<\varepsilon$ for  $x\in X$ and  $n\geq N$. Hence, $f_n$ converges uniformly to $f$.
\end{proof}

\begin{theorem}\label{transitivity-entropy}
{\rm(i)} If $(I,f)$ is   transitive, then $h(\hat f)=\infty$.

{\rm(ii)} There exists $(I,f_{0,\infty})$ such that $(I,f_{0,\infty})$ is transitive but $h(\hat f_{0,\infty})=0$.
\end{theorem}
\begin{proof}
(i) Corollary 3.6 in \cite{BC87}  ensures that $h(f)\geq {1\over 2}\log 2$. Thus, $h(\hat f)=\infty$ by Proposition 6 in \cite{Bauer}.

(ii)
Let $I=[0,1]$. First, we construct a family of functions  $F_m:I\to I$ for $m>0$. Divide $I$ into $m$ intervals $J_i\triangleq[a_i,a_{i+1}]$, $0\leq i\leq m-1$, where $a_i={i\over m}$.
For any $0\leq i\leq m-1$, put $c_i,d_i\in J_i$ with $c_i=a_i+{1\over3m}$, $d_i=a_i+{2\over3m}$, $d_{-1}=0$ and $c_m=1$. The map $F_m$ is the connect-the-dots map
such that $F_m(a_i)=a_i$, $F_m(c_i)=c_{i+1}$ and $F_m(d_i)=d_{i-1}$ for   $0\leq i\leq m-1$. Then $(I,F_m)$ is  exact for any fixed $m\geq 1$.
Next,  inductively define the maps $\{f_n\}_{n=0}^{\infty}$. Let $\mathcal{A}_{n}\triangleq\{[{{i}\over{2^{n}}},{{i+1}\over{2^{n}}}]:i=0,\cdots,2^{n}-1\}$.
Then there exists $s_1\geq 1$ such that $\underbrace{F_{1}\circ\cdots\circ F_{1}}_{s_1}(J)=I$ for  $J\in\mathcal{A}_{1}$. Denote $f_i\triangleq F_1$ for  $0\leq i\leq s_1-1$.
Assume that we have already defined $s_1<s_2<\cdots<s_n$ and $\{f_j\}_{j=0}^{s_n-1}$ such that $f_{0}^{s_k}(J)=I$ for  $J\in\mathcal{A}_{k}$ and $1\leq k\leq n$.
Let us define $s_{n+1}$ and  $\{f_j\}_{j=s_n}^{s_{n+1}-1}$.
For any $J\in\mathcal{A}_{n+1}$, there exists
$l\geq 1$ such that $\underbrace{F_{n+1}\circ\cdots\circ F_{n+1}}_{l}(f_{0}^{s_n}(J))=I$. Denote $s_{n+1}\triangleq s_n+l$ and $f_j\triangleq F_{n+1}$ for $s_n\leq j\leq s_{n+1}-1$.
By the above construction, we get that $ f_n$ converges uniformly to ${id}$ on $I$, and $f_{0}^{s_{n}}(J)=I$ for  $n\geq 1$
and $J\in\mathcal{A}_{n}$, which means that  $(I, f_{0,\infty})$ is transitive.
This was first shown in
  Theorem 12 of \cite{BO12}.
 By Lemma \ref{converge uniform}, $\hat f_n$ converges uniformly to $\hat {id}$ on $\mathcal{M}(I)$.
It then follows from Theorem E in \cite{Kolyada96} that $h(\hat f_{0,\infty})\leq h(\hat {id})=0$.
\end{proof}
\subsection{Weak mixing} In \cite{Bauer}, it was proven that weak mixing of order $2$ is equivalent between an autonomous system $(X,f)$ and $(\mathcal{M}(X),\hat{f})$.  However, as we will  show in  Theorem \ref{weak mixing order 2} and Proposition \ref{weak mixing some order all orders},  weak mixing of order $2$ of $(X,f_{0,\infty})$ can not carry over to
 $(\mathcal{M}(X),\hat{f}_{0,\infty})$ in general, while weak mixing of order $n$ of $(\mathcal{M}(X),\hat{f}_{0,\infty})$ is sufficient to prove that of a non-autonomous system $(X,f_{0,\infty})$  for  any $n\geq2$.
\begin{theorem}\label{weak mixing order 2}
There exists
$(X,f_{0,\infty})$ such that

{\rm(1)} it is weakly mixing of order $2$;

{\rm(2)}  $(\mathcal{M}(X),\hat{f}_{0,\infty})$ is not weakly mixing of order $2$.
\end{theorem}

\begin{proof}
Recall that  $\mathbb{S}^1$ is  the unit circle and  $R: \mathbb{S}^1\to\mathbb{S}^1$ is an
irrational rotation. Let $T(x)=(1/2)x+(1/2)x^{2}$ for $x\in[0,1]$. We regard  $T$ the same  with its lift to $\mathbb{S}^1$
obtained by identifying the endpoints of the interval, where $z=1\in\mathbb{S}^1$ is the point where these endpoints are ``glued" together.
Let $\{g_n\}_{n=0}^{\infty}$ be the sequence of all possible finite sequences of  $T$ and $R$; that is,
\begin{align*}
\{g_n\}_{n=0}^{\infty}=\{R, T, (R\circ R), (R\circ T), (T\circ R), (T\circ T), (R\circ R\circ R), (R\circ R\circ T),\cdots\},
\end{align*}
and
\begin{align*}
\{f_n\}_{n=0}^{\infty}&=\{g_0, g_0^{-1}, g_1, g_1^{-1}, g_2, g_2^{-1},\cdots\}.
\end{align*}
It was proved in Theorem 6 of \cite{BO12} that $(\mathbb{S}^1,f_{0,\infty})$ is weakly mixing of order $2$.

Next, we prove that $(\mathcal{M}(\mathbb{S}^1),\hat{f}_{0,\infty})$ is not weakly mixing of order $2$.
Let $a_1$, $a_2$ and $a_3$ be three points ordered clockwise in $\mathbb{S}^1$. Choose $\varepsilon_0\in (0,{1\over3})$ and $\varepsilon_1\in(0,{1\over3}-\varepsilon_0)$ small enough such that $B_d(a_i,\varepsilon_1)$ are disjointed for $i=1,2,3.$
Define $\mu={1\over3}(\delta_{a_1}+\delta_{a_2}+\delta_{a_3})$. For a fixed $j_0\in\{1,2,3\}$, we claim that if there exists $n_0\geq0$ such that
\begin{align}\label{condition-measure-intersect-non-empty}
\hat f_0^{n_0}(B_{\mathcal{P}_d}(\mu,\varepsilon_0))\cap B_{\mathcal{P}_d}(\delta_{a_{j_0}},\varepsilon_1)\neq\emptyset,
\end{align}
 then
\begin{align}\label{conclusion-measure-intersect-non-empty}
f_0^{n_0}(B_d(a_{i},\varepsilon_0))\cap B_d(a_{j_0},\varepsilon_1)\neq\emptyset, \;\;i=1,2,3.
\end{align}
We only prove $f_0^{n_0}(B_d(a_{1},\varepsilon_0))\cap B_d(a_{j_0},\varepsilon_1)\neq\emptyset$ since others can be proved similarly.
By \eqref{condition-measure-intersect-non-empty}, there exists $\nu_0\in B_{\mathcal{P}_d}(\mu,\varepsilon_0)$ such that $\hat f_0^{n_0}(\nu_0)\in B_{\mathcal{P}_d}(\delta_{a_{j_0}},\varepsilon_1)$. Then
\begin{align*}
{1\over 3}=\mu(\{a_1\})\leq \nu_0(B_d(a_1,\varepsilon_0))+\varepsilon_0\Rightarrow \nu_0(B_d(a_1,\varepsilon_0))\geq {1\over3}-\varepsilon_0,
\end{align*}
and
\begin{align*}
1=\delta_{a_{j_0}}(\{a_{j_0}\})\leq \nu_0(f_0^{-n_0}(B_d(a_{j_0},\varepsilon_1)))+\varepsilon_1\Rightarrow \nu_0(f_0^{-n_0}(B_d(a_{j_0},\varepsilon_1)))\geq 1-\varepsilon_1.
\end{align*}
Thus, we have $\nu_0(B_d(a_1,\varepsilon_0))+\nu_0(f_0^{-n_0}(B_d(a_{j_0},\varepsilon_1)))\geq {4\over3}-(\varepsilon_0+\varepsilon_1)>1$. This means that
$f_0^{n_0}(B_d(a_1,\varepsilon_0))\cap B_d(a_{j_0},\varepsilon_1)\neq\emptyset.$

Suppose that  $(\mathcal{M}(\mathbb{S}^1),\hat{f}_{0,\infty})$ is  weakly mixing of order $2$.
Then there exists $n_1\geq 0$ such that
\begin{align*}
\hat f_0^{n_1}(B_{\mathcal{P}_d}(\mu,\varepsilon_0))\cap B_{\mathcal{P}_d}(\delta_{a_{1}},\varepsilon_1)\neq\emptyset,\;\;
\hat f_0^{n_1}(B_{\mathcal{P}_d}(\mu,\varepsilon_0))\cap B_{\mathcal{P}_d}(\delta_{a_{2}},\varepsilon_1)\neq\emptyset.
\end{align*}
Then by \eqref{conclusion-measure-intersect-non-empty}, we have
\begin{align}\label{measure-intersect-nonempty-contradition}
f_0^{n_1}(B_d(a_{i},\varepsilon_0))\cap B_d(a_{j},\varepsilon_1)\neq\emptyset, \;\;i=1,2,3,j=1,2.
\end{align}
Since $R$ and $T$ are order preserving homeomorphisms, we have ${f}_{0}^{n_1}(B_d(a_{1},\varepsilon_0)), {f}_{0}^{n_1}(B_d(a_{2},\varepsilon_0))$ and $ {f}_{0}^{n_1}(B_d(a_{3},\varepsilon_0))
$ are ordered clockwise, disjoint and connected. Thus, we get by \eqref{measure-intersect-nonempty-contradition} that there exists $i_1\in\{1,2,3\}$ such that ${f}_{0}^{n_1}(B_d(a_{i_1},\varepsilon_0))\subset B_d(a_{1},\varepsilon_1)$, which contradicts that ${f}_{0}^{n_1}(B_d(a_{i_1},\varepsilon_0))\cap B_d(a_{2},\varepsilon_1)\neq\emptyset$ and $B_d(a_{1},\varepsilon_1)\cap B_d(a_{2},\varepsilon_1)=\emptyset$.
\end{proof}
On the other hand, we also get  the interrelations of weak mixing between $(X,f_{0,\infty})$ and $(\mathcal{M}(X),\hat f_{0,\infty})$ in the next
proposition.
A  technical lemma will be needed here.
\begin{lemma}\label{3.2}
Let $\mathcal{U}_{1},\cdots,\mathcal{U}_{n}$ be nonempty open subsets of $\mathcal{M}(X)$. Then there exists $k\geq 1$
such that for  $1\leq j\leq n$, there exist nonempty open subsets $U_{1}^{j},\cdots,U_{k}^{j}$ of $X$ satisfying that $\frac{1}{k}\sum\limits_{i=1}^{k}\delta_{y_i^{j}}\in \mathcal{U}_{j}$ for
$y_{i}^{j}\in U_{i}^{j}$, $1\leq i\leq k$.
\end{lemma}

\begin{proof}
By Lemma \ref{basic result} (ii), we can choose  $k\geq 1$ such that there exists  $\mu_{j}\triangleq{1\over k}\sum\limits_{i=1}^k\delta_{x_i^j}\in\mathcal{U}_{j}$ for   $1\leq j\leq n$, where $x^{j}\triangleq(x_1^{j},\cdots,x_k^{j})\in X^{k}$.
Clearly,
$\varphi_k(x^{j})=\mu_j$, where $\varphi_k$ is given by (\ref{2.3}).
Choose $\epsilon>0$ such that $B_{\mathcal{P}_{d}}(\mu_{j},\epsilon)\subset \mathcal{U}_{j}$ for  $1\leq j\leq n$.    By the continuity of $\varphi_k$,
there exists $\delta>0$ such that  $\mathcal{P}_{d}(\varphi_k(x),\varphi_k(y))<\epsilon$ for  $x,y\in X^{k}$ with $d_k(x,y)<\delta$.
Let $U_{i}^{j}\triangleq B_{d}(x_i^{j},\delta)$ for $1\leq i\leq k$, and $y^{j}\triangleq(y_1^{j},\cdots,y_k^{j})\in \prod\limits_{i=1}^kU_{i}^{j}$. Then $d_k(x^{j},y^{j})=\max\limits_{1\leq i\leq k}d(x_i^{j},y_i^{j})<\delta$. Thus, $\varphi_k(y^{j})=\frac{1}{k}\sum\limits_{i=1}^{k}\delta_{y_i^{j}}\in B_{\mathcal{P}_{d}}(\mu_j,\epsilon)\subset\mathcal{U}_j$ for
$1\leq j\leq n$.
\end{proof}

\begin{proposition}\label{weak mixing some order all orders}
If $(X,f_{0,\infty})$ is weakly mixing of all orders, then so is $(\mathcal{M}(X),\hat f_{0,\infty})$.
Conversely, if $(\mathcal{M}(X),\hat f_{0,\infty})$ is weakly mixing of some order $n$, then so is $(X,f_{0,\infty})$.
\end{proposition}

\begin{proof}
Suppose that $(X,f_{0,\infty})$ is weakly mixing of all orders. Fix $n\geq2$. Let $\mathcal{U}_j$ and $\mathcal{V}_j$,
$1\leq j\leq n$, be nonempty open subsets of $\mathcal{M}(X)$. By Lemma \ref{3.2}, there exists
$k\geq 1$ such that for any $1\leq j\leq n$, there exist nonempty open subsets $U_{1}^{j},\cdots,U_{k}^{j}$,
$V_{1}^{j},\cdots,V_{k}^{j}$ of $X$ satisfying that
$\frac{1}{k}\sum\limits_{i=1}^{k}\delta_{y_i^{j}}\in \mathcal{U}_{j}$ and $\frac{1}{k}\sum\limits_{i=1}^{k}\delta_{z_i^{j}}\in \mathcal{V}_{j}$ for $y_{i}^{j}\in U_{i}^{j}$, $z_{i}^{j}\in V_{i}^{j}$ and $1\leq i\leq k$.
Since $(X,f_{0,\infty})$ is weakly mixing of all orders, there exist $k_1\geq 1$ and $\tilde y_{i}^{j}\in U_{i}^{j}$
such that $f_{0}^{k_1}(\tilde y_{i}^{j})\in V_{i}^{j}$ for $1\leq j\leq n$ and $1\leq i\leq k$. Denote $\mu_{j}\triangleq\frac{1}{k}\sum\limits_{i=1}^{k}\delta_{\tilde y_{i}^{j}}$. Then $\mu_{j}\in\mathcal{U}_j$ and
$\hat{f}_{0}^{k_1}(\mu_{j})=\frac{1}{k}\sum\limits_{i=1}^{k}\delta_{f_{0}^{k_1}(\tilde y_{i}^{j})}\in\mathcal{V}_j$
by Lemma \ref{basic result} (iii). So,
$
\hat{f}_{0}^{k_1}(\mathcal{U}_j)\cap \mathcal{V}_j\neq\emptyset$ for all $1\leq j\leq n.
$
Hence, $(\mathcal{M}(X),\hat f_{0,\infty})$ is weakly mixing of order $n$. Since $n$ is arbitrary,
$(\mathcal{M}(X),\hat f_{0,\infty})$ is weakly mixing of all orders.

Suppose that $(\mathcal{M}(X),\hat f_{0,\infty})$ is weakly mixing of some order $n\geq2$. Let $U_1,\cdots,U_n$, $V_1,\cdots,V_n$
be nonempty open subsets of $X$. Define
$\mathcal{U}_i\triangleq\{\mu\in\mathcal{M}(X): \mu(U_i)>1/2\}$
and
$\mathcal{V}_i\triangleq\{\mu\in\mathcal{M}(X): \mu(V_i)>1/2\}$ for $1\leq i\leq n,$
which  are clearly nonempty open subsets of $\mathcal{M}(X)$.
Since $(\mathcal{M}(X),\hat f_{0,\infty})$ is weakly mixing of order $n$, there exists $k>0$ such that
 for any $1\leq i\leq n$, there exists $\mu_i\in\mathcal{U}_i$ satisfying that $\hat{f}_{0}^{k}(\mu_i)\in\mathcal{V}_i$.
Thus, $\mu_i(U_i)>1/2$ and $\mu_i(f_{0}^{-k}(V_i))=\hat{f}_{0}^{k}(\mu_i)(V_i)>1/2$, which implies that
$f_{0}^{k}(U_i)\cap V_i\neq\emptyset$ for all $1\leq i\leq n.$
Hence, $(X,f_{0,\infty})$ is weakly mixing of order $n$.
\end{proof}
\subsection{Mixing and exactness}
In this subsection,  the connections of  mixing  and exactness between $(X,f_{0,\infty})$ and $(\mathcal{M}(X),\hat f_{0,\infty})$ are studied.

\begin{lemma}\label{3.1}
{\rm(i)} If $(X,f_{0,\infty})$ is  mixing,
then $(X^{n},f_{0,\infty}^{n})$ is  mixing for  $n\geq1$.

{\rm(ii)} If $(X,f_{0,\infty})$ is mild mixing,
then $(X^{n},f_{0,\infty}^{n})$ is  mild mixing for  $n\geq1$,
where $f_{0,\infty}^{n}=\underbrace{f_{0,\infty}\times\cdots\times f_{0,\infty}}_{n}$.
\end{lemma}

\begin{proof}
(i) is straightforward to verify, and (ii) is obtained by  simple induction.
\end{proof}

Motivated by \cite{Bauer,Wu}, we prove in the following two theorems that  mixing (resp., mild mixing and  exactness) of $(X,f_{0,\infty})$ is equivalent to that of $(\mathcal{M}(X),\hat f_{0,\infty})$.
\begin{theorem}\label{mixing}
{\rm(i)} $(X,f_{0,\infty})$ is  mixing  if and only if
$(\mathcal{M}(X),\hat f_{0,\infty})$ is  mixing.

{\rm(ii)} $(X,f_{0,\infty})$ is  mild mixing if and only if
$(\mathcal{M}(X),\hat f_{0,\infty})$ is  mild mixing.
\end{theorem}

\begin{proof}
 (i) Suppose that $(X,f_{0,\infty})$ is  mixing. Let $\mathcal{U}_{1}$ and
$\mathcal{U}_{2}$ be two nonempty open subsets  of $\mathcal{M}(X)$. Lemma \ref{3.2} ensures that there exists $n\geq 1$ such that
for any $1\leq j\leq 2$, there exist $n$ nonempty open subsets $U_{1}^{j},\cdots,U_{n}^{j}$ of $X$ satisfying that $\frac{1}{n}\sum\limits_{i=1}^{n}\delta_{y_i^{j}}\in \mathcal{U}_{j}$
for  $y_{i}^{j}\in U_{i}^{j}$ and $1\leq i\leq n$.
By Lemma \ref{3.1}, $(X^{n},f_{0,\infty}^{n})$ is  mixing. Thus, there exists $N_0\geq 1$
such that for any $N\geq N_0$, we can choose $z_i\in U_{i}^{1}$ satisfying that $f_{0}^{N}(z_i)\in U_{i}^{2}$
for  $1\leq i\leq n$. Define $\mu\triangleq\frac{1}{n}\sum\limits_{i=1}^{n}\delta_{z_i}$. Then
$\mu\in\mathcal{U}_{1}$ and $\hat{f}_{0}^{N}(\mu)=\frac{1}{n}\sum\limits_{i=1}^{n}\delta_{f_{0}^{N}(z_i)}\in\mathcal{U}_{2}$, which yields $\hat{f}_{0}^{N}(\mathcal{U}_{1})\cap\mathcal{U}_{2}\neq\emptyset$ for  $N\geq N_0$.
Therefore, $(\mathcal{M}(X),\hat f_{0,\infty})$ is  mixing. The converse of (i)
is proved by a similar approach to that of Proposition \ref{transitive}.

(ii) Let $(Y,g_{0,\infty})$ be a given transitive system. Suppose that $(X,f_{0,\infty})$ is mild mixing.  Choose two nonempty open subsets  $\mathcal{O}_{1}$ and $\mathcal{O}_{2}$  of
$\mathcal{M}(X)\times Y$. Then there exist nonempty open subsets $\mathcal{U}_{i}\subset\mathcal{M}(X)$ and
$V_i\subset Y$ such that $\mathcal{U}_{i}\times V_i\subset\mathcal{O}_{i}$, $i=1,2$. By Lemma \ref{3.2}, there exists
$n\geq 1$ such that for  $1\leq j\leq 2$, there exist $n$ nonempty open subsets $U_{1}^{j},\cdots,U_{n}^{j}$
of $X$ satisfying that  $\frac{1}{n}\sum\limits_{i=1}^{n}\delta_{y_i^{j}}\in
\mathcal{U}_{j}$  for  $y_{i}^{j}\in U_{i}^{j}$ and $1\leq i\leq n$. By Lemma \ref{3.1},
 $(X^{n}\times Y,f_{0,\infty}^{n}\times g_{0,\infty})$ is transitive. Thus, there exists $N_0\geq 1$ such that
$
g_{0}^{N_0}(V_1)\cap V_2\neq\emptyset,
$
and  there exists $z_i\in U_{i}^{1}$ such that $f_{0}^{N_0}(z_i)\in U_{i}^{2}$ for  $1\leq i\leq n$.
 Let $\mu\triangleq\frac{1}{n}\sum\limits_{i=1}^{n}\delta_{z_i}$. Then $\mu\in\mathcal{U}_{1}$ and $\hat{f}_{0}^{N_0}(\mu)=\frac{1}{n}\sum\limits_{i=1}^{n}\delta_{f_{0}^{N}(z_i)}\in\mathcal{U}_{2}$. Thus,
$
(\hat{f}_{0}^{N_0}\times  g_{0}^{N_0})(\mathcal{O}_{1})\cap\mathcal{O}_{2}\neq\emptyset.
$
This proves that $(\mathcal{M}(X),\hat f_{0,\infty})$ is mild mixing.

Suppose that $(\mathcal{M}(X),\hat f_{0,\infty})$ is mild mixing.
It suffices to show that there exists  $N_0\geq 1$ such that
$ (f_{0}^{N_0}\times g_{0}^{N_0})(U_{1}\times V_1)\cap (U_{2}\times V_2)\neq\emptyset$ for
any nonempty open subsets $U_{1},U_{2}\subset X$
and $V_1,V_2\subset Y$. Define
$
\mathcal{U}_{i}\triangleq\{\mu\in\mathcal{M}(X): \mu(U_i)>1/2\}$, $ i=1,2.
$
Since $(\mathcal{M}(X),\hat f_{0,\infty})$ is mild mixing, there exists $N_0\geq 1$ such that
$
 g_{0}^{N_0}(V_1)\cap  V_2\neq\emptyset
$
and there exists $\mu\in \mathcal{U}_{1}$ such that $\hat{f}_{0}^{N_0}(\mu)\in \mathcal{U}_{2}$. Since $\mu(U_1)>1/2$ and $\hat{f}_{0}^{N_0}(\mu)(U_2)=\mu\big(f_{0}^{-N_0}(U_2)\big)>1/2$,
we have $ (f_{0}^{N_0}\times g_{0}^{N_0})(U_{1}\times V_1)\cap (U_{2}\times V_2)\neq\emptyset$.
\end{proof}

\begin{theorem}\label{exactness}
$(X,f_{0,\infty})$ is  exact if and only if $(\mathcal{M}(X),\hat f_{0,\infty})$ is  exact.
\end{theorem}

\begin{proof}
Suppose that $(X,f_{0,\infty})$ is  exact. Fix any nonempty open subset $\mathcal{U}$ of $\mathcal{M}(X)$.
Lemma \ref{basic result} (ii) shows that there exists $m\geq 1$ such that $\mu_0\triangleq\frac{1}{m}\sum\limits_{i=1}^{m}\delta_{x_i}\in\mathcal{U}$ for some $x_1,\cdots,x_m\in X$. Choose
$\epsilon>0$ such that $\bar{B}_{\mathcal{P}_d}(\mu_0,\epsilon)\subset\mathcal{U}$.
Since $(X,f_{0,\infty})$ is  exact, there exists $N_0\geq 1$ such that for  $1\leq i\leq m$,
\begin{align}\label{exact}
f_{0}^{N}\big(B_{d}(x_i,\epsilon)\big)=X,\;N\geq N_0.
\end{align}
Fix $N\geq N_0$. For any $\nu\in\mathcal{M}_\infty(X)$, there exist $n\geq1$ and $y_1,\cdots,y_n\in X$ such that $\nu=\frac{1}{n}\sum\limits_{j=1}^{n}\delta_{y_j}$. It follows from (\ref{exact}) that for  $1\leq i\leq m$ and $1\leq j\leq n$, there exists
$z_{j}^{i}\in B_{d}(x_i,\epsilon)$ such that $f_{0}^{N}(z_{j}^{i})=y_j$. Define
$\mu\triangleq\frac{1}{mn}\sum\limits_{i=1}^{m}\sum\limits_{j=1}^{n}\delta_{z_{j}^i}$. By Lemma \ref{basic result} (iii), we have
\begin{align}\label{ex}
\hat{f}_{0}^{N}(\mu)
=\frac{1}{mn}\sum_{i=1}^{m}\sum_{j=1}^{n}\delta_{f_{0}^{N}(z_{j}^i)}
=\frac{1}{n}\sum_{j=1}^{n}\delta_{y_j}=\nu.
\end{align}
Since
$z_{j}^{i}\in B_{d}(x_i,\epsilon)$ for  $1\leq i\leq m$ and $1\leq j\leq n$, we have
\begin{align*}
\mu(A^{\epsilon})+\epsilon=\frac{1}{mn}\sum_{i=1}^{m}\sum_{j=1}^{n}\delta_{z_{j}^i}(A^{\epsilon})+\epsilon\geq \frac{1}{m}\sum_{i=1}^{m}\delta_{x_i}(A)=\mu_0(A)
\end{align*}
for  $A\in \mathcal{B}(X)$.
Thus, $\mathcal{P}_{d}\left(\mu,\mu_0\right)<\epsilon$. This, along with (\ref{ex}), gives $\nu\in\hat{f}_{0}^{N}\big(B_{\mathcal{P}_{d}}(\mu_0,\epsilon)\big)$. Thus,
\begin{align*}
\mathcal{M}_\infty(X)\subset\hat{f}_{0}^{N}\big(B_{\mathcal{P}_{d}}
(\mu_0,\epsilon)\big).
\end{align*}
Lemma \ref{basic result} (ii) implies that
\begin{align*}
\mathcal{M}(X)=\overline{\mathcal{M}_\infty(X)}\subset\overline{\hat{f}_{0}^{N}
\big(B_{\mathcal{P}_{d}}(\mu_0,\epsilon)\big)}=\hat{f}_{0}^{N}\big(\bar{B}_{\mathcal{P}_{d}}
(\mu_0,\epsilon)\big)\subset\hat{f}_{0}^{N}(\mathcal{U}).
\end{align*}
Consequently, $\hat{f}_{0}^{N}(\mathcal{U})=\mathcal{M}(X)$ for $N\geq N_0$.
Hence, $(\mathcal{M}(X),\hat f_{0,\infty})$ is  exact.

Let $(\mathcal{M}(X),\hat f_{0,\infty})$ be  exact. Suppose that $(X,f_{0,\infty})$
is not  exact. Then there exists a nonempty open subset $U$ of $X$ such that for any $N\geq 1$,
there exists $n_0\geq N$ such that $f_{0}^{n_0}(U)\subsetneqq X$. Let $x\in U$ and choose
$0<r<1/2$ such that $\bar{B}_d(x,r)\subset U$. Clearly, $X\setminus f_{0}^{n_0}\big(\bar{B}_d(x,r)\big)\neq\emptyset$.
Let $y\in X\setminus f_{0}^{n_0}\big(\bar{B}_d(x,r)\big)$ and define $\eta\triangleq\min\left\{1/2, d\big(y, f_{0}^{n_0}(\bar{B}_d(x,r))\big)\right\}>0.$
Lemma \ref{basic result} (ii) ensures that $\mathcal{D}\triangleq B_{\mathcal{P}_{d}}(\delta_x,r)\cap\mathcal{M}_{\infty}(X)\neq\emptyset$.
Choose $\mu\triangleq\frac{1}{k}\sum\limits_{i=1}^{k}\delta_{y_i}\in\mathcal{D}$
  for some $k\geq 1$ and $y_1,\cdots,y_k\in X$.
Then
\begin{align*}
1=\delta_{x}(\{x\})\leq\mu(B_d(x,r))+r={1\over k}\sum\limits_{i=1}^k\delta_{y_i}(B_d(x,r))+r={1\over k}\sharp\{1\leq i\leq k:y_i\in B_d(x,r)\}+r.
\end{align*}
Since $y_i\in B_d(x,r)$ implies $f_0^{n_0}(y_i)\notin B_d(y,\eta)$, we have
\begin{align*}
\hat{f}_{0}^{n_0}(\mu)(B_d(y,{\eta}))+\eta
\leq\frac{1}{k}\sharp\{1\leq i\leq k: y_i\notin B_d(x,{r})\}+\eta\leq r+\eta<\delta_{y}(\{y\}).
\end{align*}
This proves $\mathcal{P}_{d}\big(\delta_{y},\hat{f}_{0}^{n_0}(\mu)\big)\geq \eta$. Hence,
$\mathcal{P}_{d}(\delta_{y},\hat{f}_{0}^{n_0}(\mathcal{D}))\geq\eta$.
Since $\overline{\mathcal{M}_\infty(X)}=\mathcal{M}(X)$ and $B_{\mathcal{P}_d}(\delta_x,r)$ is open, we have
$\hat f_0^{n_0}(\bar{B}_{\mathcal{P}_d}(\delta_x,r))=\overline{\hat f_0^{n_0}(\mathcal{D})}$ and  thus
\begin{align*}
\mathcal{P}_{d}\left(\delta_{y},\hat{f}_{0}^{n_0}\big(\bar{B}_{\mathcal{P}_{d}}(\delta_x,r)\big)\right)
=\mathcal{P}_{d}\left(\delta_{y},\overline{\hat{f}_{0}^{n_0}(\mathcal{D})}\right)
\geq\eta.
\end{align*}
This implies that $\hat{f}_{0}^{n_0}\big(B_{\mathcal{P}_{d}}(\delta_x,r)\big)\neq \mathcal{M}(X)$, which is a contradiction.
Therefore, $(X,f_{0,\infty})$ is  exact.
\end{proof}

\section{Chain mixing, chain transitivity and shadowing}

\subsection{Chain mixing and chain transitivity}
Let $\delta>0$. Recall that a $\delta$-pseudo orbit of $(X,f_{0,\infty})$ is a finite or infinite sequence $\{x_0,x_1,\cdots\}$ such
that $d(f_n(x_n),x_{n+1})<\delta$ for all $n=0,1\cdots$. A finite $\delta$-pseudo orbit $\{x_0,x_1,\cdots,x_k\}$ is also called a
$\delta$-chain from $x_0$ to $x_k$ with length $k$. $(X,f_{0,\infty})$ is  chain mixing if, for any $\varepsilon>0$ and
any $x,y\in X$, there exists $N\geq 1$ such that for any $k\geq N$, there exists an $\varepsilon$-chain from $x$
to $y$ with length $k$; it is chain transitive if, for any $x,y\in X$ and any $\varepsilon>0$, there exists an
$\varepsilon$-chain from $x$ to $y$; it is chain weakly mixing of all orders if $(X^{n},f_{0,\infty}^{n})$ is chain transitive
for all $n\geq 1$; it is chain exact if, for any $\varepsilon>0$ and any nonempty open subset $U$ of $X$, there exists
$N\geq 1$ such that for any $x\in X$, there exist $y\in U$ and an $\varepsilon$-chain from $y$ to $x$ with length $N$.

The first lemma reveals basic relations of these chain properties.

\begin{lemma}\label{chain}
{\rm(i)} If $(X,f_{0,\infty})$ is chain mixing, then it is chain weakly mixing of all orders.

{\rm(ii)} If $(X,f_{0,\infty})$ is weakly mixing of all orders, then it is chain exact.

{\rm(iii)} If $(X,f_{0,\infty})$ is chain exact, then it is chain transitive.
\end{lemma}

\begin{proof}
(i) Let $n\geq 1$, $\varepsilon>0$ and $x=(x_1,\cdots, x_n),y=(y_1,\cdots, y_n)\in X^{n}$.
Since $(X,f_{0,\infty})$ is chain mixing, there exists $N\geq 1$ such that for any $k\geq N$ and $1\leq i\leq n$, there exists an ${\varepsilon}$-chain $\{x_i=x_{i}^{0},x_{i}^{1},\cdots,x_{i}^{k}=y_i\}$ with length $k$.
Then $\{x=(x_{1}^{0},\cdots,x_{n}^{0}),(x_{1}^{1},\cdots,x_{n}^{1}),\cdots,$ $(x_{1}^{k},\cdots,x_{n}^{k})=y\}$ is an $\varepsilon$-chain of
$(X^{n},f_{0,\infty}^{n})$ with length $k$. In fact, for any $0\leq j\leq k-1$,
\begin{align*}
d_{n}\big((f_{j}(x_{1}^{j}),\cdots,f_{j}(x_{n}^{j})),(x_{1}^{j+1},\cdots,x_{n}^{j+1})\big)
=\max_{1\leq i\leq n}d(f_{j}(x_{i}^{j}),x_{i}^{j+1})<\varepsilon.
\end{align*}
Hence, $(X,f_{0,\infty})$ is chain weakly mixing of order $n$. Therefore,
it is chain weakly mixing of all orders.

(ii) Let $\varepsilon>0$ and $U$ be a nonempty open subset of $X$. Since $X$ is compact, there exist $x_1,\cdots,x_k\in X$
 such that $X=\bigcup_{i=1}^{k}B_d(x_i,{\varepsilon\over 2})$. Fix $y\in U$. Since $(X^k, f_{0,\infty}^k)$
is chain transitive, there exists ${\varepsilon\over2}$-chain $\{(y,\cdots,y)=(z_1^0,\cdots, z_k^0),(z_1^1,\cdots, z_k^1),\cdots,
(z_1^l,\cdots,z_k^l)=(x_1,\cdots,x_k)\}$ with length $l$. For any $x\in X$, there exists $1\leq j\leq k$
such that $x\in B_d(x_j,{\varepsilon\over 2})$. Then
\begin{align*}
d(f_{l-1}(z_j^{l-1}),x)\leq d(f_{l-1}(z_j^{l-1}),x_j)+d(x_j,x)<\varepsilon.
\end{align*}
Thus, $\{y=z_j^0,z_j^1,\cdots,z_j^{l-1},x\}$ is an $\varepsilon$-chain from $y$ to $x$ with length $l$.

(iii) Fix $x,y\in X$ and $\varepsilon>0$. There exists $\delta>0$ such that $d(f_0(z_1),f_0(z_2))<{\varepsilon\over 2}$
for any $z_1,z_2\in X$ with $d(z_1,z_2)<\delta$. Since $X$ is compact, there exist $y_1,\cdots,y_k\in X$ such that
$X=\bigcup_{i=1}^k B_d(y_i,{\delta/2})$. Then $x\in B_d(y_{i_0},{\delta/2})$ for some $1\leq i_0\leq k$.
Since $(X,f_{0,\infty})$ is chain exact, there exist $N_0\geq 1$, $\hat x\in B_d(y_{i_0},{\delta/2})$ and
$\{a_i\}_{i=1}^{N_0-1}\subset X$ such that $\{\hat x=a_0,a_1,\cdots, a_{N_0}=y\}$ is an ${\varepsilon\over 2}$-chain. Since
$d(x,\hat x)\leq d(x,y_{i_0})+d(y_{i_0},\hat x)<\delta$, we have
$d(f_0( x),f_0(\hat x))<{\varepsilon/2}$. Then
\begin{align*}
d(f_0(x),a_1)\leq d(f_0(x), f_0(\hat x))+d(f_0(\hat x),a_1)<\varepsilon.
\end{align*}
Thus, $\{x,a_1,\cdots, a_{N_0}=y\}$ is an $\varepsilon$-chain from $x$ to $y$.
\end{proof}

For autonomous systems, it was proved in Theorem 13 of \cite{Bernardes} that  $(\mathcal{M}(X),\hat f)$ is chain mixing if $f$ is a homeomorphism. We will show that this result is true only if $f$ is  surjective in Theorem \ref{chain mixing}, and furthermore, it is generalized to non-autonomous systems. The key observation is the following lemma.

\begin{lemma}\label{surjective}
Let $f: X\to X$ be a map. Then
 $f$ is surjective if and only if $\hat f$ is surjective.
\end{lemma}

\begin{proof}
Suppose that $f$ is surjective.
Let $\mu\in \mathcal{M}(X)$. Lemma \ref{basic result} (ii) implies that for any $n\geq 1$, there exist
$k_n\geq 1$ and $\{x_{i,n}\}_{i=1}^{k_n}\subset X$ such that
$\frac{1}{k_{n}}\sum\limits_{i=1}^{k_n}\delta_{x_{i,n}}\in B_{\mathcal{P}_d}(\mu,\frac{1}{2^{n}})$.
Clearly, $\lim\limits_{n\to\infty}\frac{1}{k_{n}}\sum\limits_{i=1}^{k_n}\delta_{x_{i,n}}=\mu$.
Since $f$ is surjective, there exists $y_{i,n}\in X$ such that $f(y_{i,n})=x_{i,n}$ for  $1\leq i\leq k_n$.
Then by Lemma \ref{basic result} (iii),
\begin{align*}
\lim\limits_{n\to\infty}\hat f\left(\frac{1}{k_{n}}\sum\limits_{i=1}^{k_n}\delta_{y_{i,n}}\right)=
\lim\limits_{n\to\infty}\frac{1}{k_{n}}\sum\limits_{i=1}^{k_n}\hat f(\delta_{y_{i,n}})=
\lim\limits_{n\to\infty}\frac{1}{k_{n}}\sum\limits_{i=1}^{k_n}\delta_{x_{i,n}}=\mu.
\end{align*}
Since $\mathcal{M}(X)$ is compact, up to a subsequence, there exists
$\nu\in\mathcal{M}(X)$ such that $\lim\limits_{n\to\infty}\frac{1}{k_{n}}\sum\limits_{i=1}^{k_{n}}\delta_{y_{i,{n}}}$ $=\nu$.
Then
\begin{align*}
\hat f(\nu)=\lim\limits_{n\to\infty}\hat f(\frac{1}{k_{n}}\sum\limits_{i=1}^{k_{n}}\delta_{y_{i,n}})=\mu.
\end{align*}
Therefore, $\hat f$ is surjective.

Suppose that $\hat f$ is surjective. Let $y\in X$. Then there exists $\mu\in\mathcal{M}(X)$ such that $\hat f (\mu)=\delta_y$.
Thus, $\mu(f^{-1}(\{y\}))=\hat f (\mu)(\{y\})=\delta_y(\{y\})=1$, and  $f^{-1}(\{y\})\neq\emptyset$. Hence, $f$ is  surjective.
\end{proof}
\begin{theorem}\label{chain mixing}
Let $f_n: X\to X$ be a surjective map for $n\geq0$. Then $(\mathcal{M}(X),\hat f_{0,\infty})$ is chain mixing.
\end{theorem}

\begin{proof}
It suffices to show that for any $\varepsilon\in (0,1)$ and any $\mu,\nu\in \mathcal{M}(X)$, there exists $N\geq 1$
such that for any $k\geq N$, there exists an $\varepsilon$-chain $\{\mu=\mu_0,\mu_1,\cdots,\mu_k=\nu\}$ with length $k$.
Choose $N\geq 1$ such that $1\in({(N-1)\varepsilon/2},{N\varepsilon/2}]$. Fix $k\geq N$.
By Lemma \ref{surjective}, $\hat f_0^k$ is surjective and thus there exists $\nu_*\in \mathcal{M}(X)$ such that $\hat f_0^k(\nu_*)=\nu$.
By setting $\mu_1=(1-{\varepsilon/2})\hat f_0(\mu_0)+\varepsilon\hat f_0(\nu_*)/2$, we get $\hat f_0(\mu_0)\in \bar{B}_{\mathcal{P}_d}(\mu_1,{\varepsilon/2})$ by Lemma \ref{basic result} (iv). Define
$\mu_2=(1-{\varepsilon})\hat f^2_0(\mu_0)+{\varepsilon}\hat f^2_0(\nu_*)$. Again by Lemma \ref{basic result} (iv), we have
\begin{align*}
\mathcal{P}_d(\mu_2,\hat f_1(\mu_1))=\mathcal{P}_d\left((1-{\varepsilon})\hat f^2_0(\mu_0)+{\varepsilon}\hat f^2_0(\nu_*),\left(1-{\varepsilon\over 2}\right)\hat f^2_0(\mu_0)+{\varepsilon\over 2}\hat f^2_0(\nu_*)\right)\leq {\varepsilon\over2},
\end{align*}
which gives $\hat f_1(\mu_1)\in \bar{B}_{\mathcal{P}_d}(\mu_2,{\varepsilon/2})$. By induction with setting
\begin{align*}
\mu_n=(1-{n\varepsilon/2})\hat f^n_0(\mu_0)+n\varepsilon\hat f^n_0(\nu_*)/2,\;1\leq n\leq N-1,
\end{align*}
we get $\hat f_{n-1}(\mu_{n-1})\in \bar{B}_{\mathcal{P}_d}(\mu_n,{\varepsilon/2})$. Define $\mu_N=\hat f_0^N(\nu_*)$.
By the choice of $N$, we get
\begin{align*}
\hat f_{N-1}(\mu_{N-1})(K)&=\left(1-{N\varepsilon\over 2}\right)\hat f^N_0(\mu_0)(K)+{\varepsilon\over 2}\hat f^N_0(\mu_0)(K)+{(N-1)\varepsilon\over 2} \mu_N(K)\leq{\varepsilon\over 2}+\mu_N(K^{\varepsilon\over2})
\end{align*}
for $K\in\mathcal{B}(X)$, which means that $\hat f_{N-1}(\mu_{N-1})\in \bar{B}_{\mathcal{P}_d}(\mu_N,{\varepsilon/2})$. Set
$\mu_{n+1}=\hat f_{n}(\mu_n)$ for $N\leq n\leq k-1$. Then
\begin{align*}
\mu_k=\hat f_N^{k-N}(\mu_N)=\hat f_N^{k-N}(\hat f_0^N(\nu_*))=\nu
\end{align*}
and
$\hat f_{n}(\mu_{n})\in {B_{\mathcal{P}_d}(\mu_{n+1},{\varepsilon})}$ for $0\leq n\leq k-1$. Hence, $\{\mu=\mu_0,\mu_1,\cdots,\mu_k=\nu\}$
is an $\varepsilon$-chain with length $k$. Therefore, $(\mathcal{M}(X),\hat f_{0,\infty})$ is chain mixing.
\end{proof}

Theorem \ref{chain mixing} has two applications.
The first one is a combination with Lemma \ref{chain}.

\begin{corol}\label{all}
Let $f_n: X\to X$ be a surjective map for $n\geq0$. Then $(\mathcal{M}(X), \hat f_{0,\infty})$  is chain mixing, chain
 weakly mixing of all orders, chain exact, and chain transitive.
\end{corol}
Therefore,
the surjective condition on $f_n$ brings fruitful chain properties for $(\mathcal{M}(X), \hat f_{0,\infty})$.  However, as is shown in the next remark, it is not the case for hyperspace.
\begin{remark}
Let $X=\{a,b\}$ with the discrete metric, $f(a)=b$ and $f(b)=a$. Clearly, $(X,f)$ is chain transitive. Since $f$ is surjective,
$(\mathcal{M}(X),\hat{f})$ is also chain transitive by Corollary \ref{all}. However, it was shown by
Example 3 in \cite{FGPR15} that in hyperspace, $(\mathcal{K}(X),\bar{f})$ is not chain transitive.
\end{remark}

The other application of Theorem \ref{chain mixing} is  as follows.

\begin{corol}\label{adschain}
If $(X,f)$ is chain transitive, then $(\mathcal{M}(X), \hat f)$  is chain mixing, chain
 weakly mixing of all orders, chain exact, and chain transitive.
\end{corol}
\begin{proof}
Since $(X,f)$ is chain transitive, Proposition 2.6 in \cite{K13} ensures that $f$ is surjective. Then we get the conclusion from Theorem  \ref{chain mixing} and  Lemma \ref{chain}.
\end{proof}
It is natural to consider whether the converse of Corollary \ref{adschain} is true. The following two counterexamples
confirm that chain mixing of $(\mathcal{M}(X),\hat f)$ does not imply chain mixing, not even chain transitivity of $(X,f)$.
The first example is for a surjective map, and the second one is for a homeomorphism.

\begin{example}\label{surjective} Let  $I=[0,1]$ and
\begin{align*}
		f(x)=\left\{\begin{array}{ll}
		0,& x\in [0,{1\over2}],\\
		2x-1, & x\in ({1\over2},1].
		\end{array}\right.
	\end{align*}
 Clearly,
$f: I\to I$ is continuous and surjective. Thus, $(\mathcal{M}(I),\hat{f})$ is
chain mixing by Theorem \ref{chain mixing}. However,  $(I,f)$ is not chain transitive. In fact, by choosing $\delta_0\in(0,{1\over2})$, $x_0=0$
and $y_0={2\over3}$, it can be verified that for any  $\delta_0$-pseudo orbit $\{x_n\}_{n=0}^\infty$ from  $x_0$, $\{x_n\}_{n=0}^\infty\subset[0,{1\over2}]$.  Indeed, $x_1\in B_d(f(x_0),\delta_0)=B_d(0,\delta_0)\subset[0,{1\over2}]$. If $x_{n}\in[0,{1\over2}]$ for some $n>0$, then $f(x_n)=0$ and $x_{n+1}\in B_d(f(x_n),\delta_0)\subset[0,{1\over2}]$. Thus, $\{x_n\}_{n=0}^\infty\subset[0,{1\over2}]$. However, $y_0={2\over3}$, which means that there are no $\delta_0$-chains from $x_0$ to $y_0$.
\end{example}

\begin{example}\label{homeomorphism} Let $I=[0,1]$, and
 \begin{align*}
 f(x)=x^{2},\;\; x\in I.
 \end{align*}
Then $f: I\to I$ is a homeomorphism. Thus, $(\mathcal{M}(I),\hat{f})$ is chain mixing by Theorem \ref{chain mixing}.
However, $(I,f)$ is not chain transitive. Let $\delta_{0}={1\over4}$, $x_0={1\over2}$ and $y_0=1$. We claim that
for any  $\delta_0$-pseudo orbit $\{x_n\}_{n=0}^\infty$ from  $x_0$, $\{x_n\}_{n=0}^\infty\subset[0,{1\over2})$.
Clearly, $x_{1}\in B_d(f(x_0),\delta_{0})=(0,{1\over2})$. Suppose that $x_{n}\in[0,{1\over2})$ for some $n>0$.
Then $f(x_n)\in[0,{1\over4})$, and thus $x_{n+1}\in B_d(f(x_n),\delta_{0})$ $\subset[0,{1\over2})$.
Hence, $\{x_n\}_{n=0}^\infty\subset[0,{1\over2})$. Since $y_0=1$, there are no $\delta_{0}$-chains
from $x_0$ to $y_0$.
\end{example}

\subsection{Shadowing and specification}
We first recall some  concepts related to the shadowing property.
$(X,f_{0,\infty})$  has the shadowing property if, for every $\varepsilon>0$, there exists $\delta>0$ such that  every $\delta$-pseudo orbit
$\{x_0,x_1,\cdots\}$ is $\varepsilon$-shadowed by some point in $X$ (i.e. there exists $y\in X$ such that
$d(f_0^n(y),x_n)<\varepsilon$ for all $n=0,1\cdots$). A sequence $\{x_i\}_{i=0}^\infty \subset X$ is  a $\delta$-average-pseudo orbit of $(X,f_{0,\infty})$
 for some $\delta>0$ if there is
$N>0$ such that for any  $n\geq N$ and $k\geq 0$,
${1\over n}
\sum\limits_{i=0}^{n-1}
d(f_{i+k}(x_{i+k}), x_{i+k+1})<\delta$.
$(X,f_{0,\infty})$ has the average-shadowing property if, for any $\varepsilon>0$, there is $\delta>0$ such that every $\delta$-average-pseudo
orbit $\{x_i\}_{i=0}^\infty$ is $\varepsilon$-shadowed in average by some point $z\in X$; that is,
$\limsup\limits_{n\to\infty}
{1\over n}
\sum\limits_{i=0}^{n-1}
d(f_0^i(z), x_i)<\varepsilon.$

It was proved in Theorem 14 and Remark 15 of \cite{Bernardes} that weak shadowing  of  $(\mathcal{M}(X), \hat f)$ implies  transitivity of $(X, f)$ if $f$ is homeomorphism. The next result indicates that  shadowing  of  $(\mathcal{M}(X), \hat f_{0,\infty})$ implies  mixing of $(X, f_{0,\infty})$ if $f_n$ is surjective for $n\geq0$.

\begin{theorem}\label{mixingshadowing}
Let $f_n: X\to X$ be a surjective map for $n\geq0$. If $(X,  f_{0,\infty})$ is not  mixing,
then $(\mathcal{M}(X), \hat f_{0,\infty})$ does not have  shadowing.
\end{theorem}

\begin{proof}
Suppose, on the contrary, that $(\mathcal{M}(X), \hat f_{0,\infty})$ has shadowing. Let $\mathcal{U}$ and $\mathcal{V}$ be two
nonempty open subsets of $\mathcal{M}(X)$. Fix $\mu\in\mathcal{U}$ and $\nu\in\mathcal{V}$. Choose $\varepsilon_0>0$
such that $B_{\mathcal{P}_d}(\mu,\varepsilon_0)\subset\mathcal{U}$ and $B_{\mathcal{P}_d}(\nu,\varepsilon_0)\subset\mathcal{V}$.
Then there exists $\delta_0>0$ such that every $\delta_0$-pseudo orbit in $\mathcal{M}(X)$ is $\varepsilon_0$-shadowed.
By Theorem \ref{chain mixing}, there exists $N\geq 1$ such that for any $k\geq N$, there exists a $\delta_0$-chain $\{\mu=\mu_0,\mu_1,\cdots,\mu_k=\nu\}$. So, there exists $\eta\in\mathcal{M}(X)$ such that
$\mathcal{P}_d(\hat f_0^n(\eta),\mu_n)<\varepsilon_0$ for all $0\leq n\leq k$. In particular,
$\eta\in B_{\mathcal{P}_d}(\mu,\varepsilon_0)\subset \mathcal{U}$ and $\hat f_0^k(\eta)\in B_{\mathcal{P}_d}(\nu,\varepsilon_0)\subset \mathcal{V}$. Hence, $(\mathcal{M}(X),\hat f_{0,\infty})$ is  mixing, and so is $(X, f_{0,\infty})$ by Theorem \ref{mixing}, which is a contradiction.
\end{proof}

As an application of Theorem \ref{mixingshadowing}, it can be shown that   shadowing of $(X,f)$ is not inherited by $(\mathcal{M}(X),\hat f)$ in general.

\begin{theorem}\label{discrete-example}
There exists $(X,f)$ such that it has shadowing but
$(\mathcal{M}(X),\hat f)$ does not have shadowing.
\end{theorem}

\begin{proof}
Let $X=\{a,b\}$ with the discrete metric $d$ and $f(a)=b$, $f(b)=a$.
Let $\epsilon>0$ and $\tilde\delta=1/2$. The $\tilde\delta$-pseudo orbits of $(X,f)$ are exactly $\{a,b,a,b,\cdots\}$ and $\{b,a,b,a,\cdots\}$.
Both of them are the true orbits of $(X,f)$. Thus, $(X,f)$ has shadowing.
Since $(X,f)$ is not  mixing, it follows  from Theorem \ref{mixingshadowing} that $(\mathcal{M}(X),\hat f)$ does not have shadowing.

Then,  a direct proof is provided for no shadowing  of $(\mathcal{M}(X),\hat f)$.
 It is straightforward to show that
\begin{align*}
\mathcal{M}(X)=\{\alpha\delta_{a}+\beta\delta_{b}: \;\alpha+\beta=1, \;0\leq \alpha,\beta\leq1\},
\end{align*}
and all the points in $\mathcal{M}(X)$ are periodic points with period 2.
Now, we show that $(\mathcal{M}(X),\hat{f})$ does not have shadowing.
Let $d_0=\mathcal{P}_{d}(\delta_a,(\delta_a+\delta_b)/2)>0$ and $\epsilon_0=d_0/2$.
 Choose $n_0\in\mathbf{Z^{+}}\cap({1\over\delta}-2,{1\over\delta}+2)$ for any $\delta\in (0,{1\over2})$.
Then we  construct a $\delta$-pseudo orbit $\{\nu_n\}_{n=0}^{\infty}$ such that for any $\mu\in\mathcal{M}(X)$, there exists $N_0\in\mathbf{Z}^+$ satisfying $d(\hat f_0^{N_0}(\mu),\nu_{N_0})\geq \varepsilon_0$. Define
\begin{align*}
		\nu_n=\left\{\begin{array}{ll}
		(1-{n\delta\over 2})\delta_a+{n\delta\over2}\delta_b,& \textrm{ if}\; 0\leq n\leq n_0 \textrm{ and } n \textrm{ is even};\\
		(1-{n\delta\over 2})\delta_b+{n\delta\over2}\delta_a, & \textrm{ if}\; 0\leq n\leq n_0 \textrm{ and } n \textrm{ is odd};\\
{1\over 2}(\delta_a+\delta_b),& \textrm{ if}\;  n>n_0.
		\end{array}\right.
	\end{align*}
By Lemma \ref{basic result} (iv) and the choice of $n_0$, we have  $d(\hat f_n(\nu_n),\nu_{n+1})<\delta$ for all $n\geq 0$, and thus $\{\nu_n\}_{n=0}^\infty$ is a $\delta$-pseudo orbit.
Fix any $\mu=\alpha_0\delta_{a}+\beta_0\delta_{b}\in\mathcal{M}(X)$, where $0\leq\alpha_0,\beta_0\leq1$.
Since $\mu$ is a $2$-periodic point, there exists $N_0\geq 1$ such that $\nu_{N_0}=(\delta_a+\delta_b)/2$
and $\hat{f}^{N_0}(\mu)=\mu$. If $\mathcal{P}_{d}(\mu,\delta_a)<\epsilon_0$, then
\begin{align*}
\mathcal{P}_{d}(\hat{f}^{N_0}(\mu),\nu_{N_0})=\mathcal{P}_{d}\left(\mu,\frac{1}{2}(\delta_a+\delta_b)\right)
\geq\mathcal{P}_{d}\left(\delta_a,\frac{1}{2}(\delta_a+\delta_b)\right)-\mathcal{P}_{d}(\mu,\delta_a)>\epsilon_0,
\end{align*}
which implies that the $\delta$-pseudo orbit $\{\nu_n\}_{n=0}^{\infty}$ cannot be $\epsilon_0$-shadowed
by any point $\mu\in\mathcal{M}(X)$. Hence, $(\mathcal{M}(X),\hat{f})$ does not have shadowing.
\end{proof}
The converse of Theorem \ref{discrete-example} is still open.
In contrast to shadowing, it will be shown that specification property (resp., property $P$) of $(X,f_{0,\infty})$ implies those of $(\mathcal{M}(X),\hat{f}_{0,\infty})$.
Recall that $(X,f_{0,\infty})$ has specification property if, for any $\epsilon>0$, there exists $M_{\epsilon}\in
\mathbf{Z^{+}}$ such that for any $k\geq2$, any $x_1,\cdots,x_k\in X$, and any $2k$ non-negative integers $a_1\leq
b_1<a_2\leq b_2<\cdots<a_k\leq b_k$ with $a_{i}-b_{i-1}\geq M_{\epsilon}$, $2\leq i\leq k$, there exists $z\in X$
such that for all $1\leq i\leq k$, it satisfies $d(f^{n}_{0}(z),f^{n}_{0}(x_i))\leq\epsilon$, $a_i\leq n\leq b_i$.
$(X,f_{0,\infty})$  has property $P$ if, for any two nonempty open subsets $U_1,U_2\subset X$, there exists $N\geq 1$
such that for any $k\geq2$ and any sequence $s=(s_1,s_2,\cdots,s_k)\in \Sigma_{2}^{k}$, there exists $x\in X$ satisfying that
$x\in U_{s_1}$, $f^{N}_{0}(x)\in U_{s_2}$, $\cdots$, $f^{(k-1)N}_{0}(x)\in U_{s_k}$, where $\Sigma_{2}^{k}
=\{\alpha=(a_1,\cdots,a_k): a_i\in\{1,2\}, 1\leq i\leq k\}$.

\begin{proposition}\label{specification property-property P}
{\rm(i)} If $(X,f_{0,\infty})$ has specification property, then so does $(\mathcal{M}(X),\hat{f}_{0,\infty})$.

{\rm(ii)}
If $(X,f_{0,\infty})$ has property $P$, then so does $(\mathcal{M}(X),\hat{f}_{0,\infty})$.
\end{proposition}

\begin{proof}  (i)
Let $\epsilon>0$. Then there exists $M_{\epsilon/2}\geq 1$ such that $(X,f_{0,\infty})$ satisfies the definition
of specification property. Fix any $k\geq2$, any $\mu_1,\cdots,\mu_k\in\mathcal{M}(X)$, and any $2k$ non-negative integers
$a_1\leq b_1<a_2\leq b_2<\cdots<a_k\leq b_k$ with $a_{i}-b_{i-1}\geq M_{\epsilon/2}$, $2\leq i\leq k$.
By the uniform continuity of $\hat{f}_{0}^{m}$, there exists
$\eta>0$ such that for any $\mu,\nu\in\mathcal{M}(X)$,
\begin{align}\label{651}
\mathcal{P}_{d}(\mu,\nu)<\eta\Rightarrow\mathcal{P}_{d}(\hat{f}_{0}^{m}(\mu),\hat{f}_{0}^{m}(\nu))<\epsilon/2, \;1\leq m\leq b_k.
\end{align}
By Lemma \ref{basic result} (ii), there exists $n_0\geq 1$ such that  $\nu_i=\frac{1}{n_0}\sum\limits_{j=1}^{n_0}\delta_{x_j^{(i)}}\in B_{\mathcal{P}_{d}}(\mu_i,\eta)$, $1\leq i\leq k$,
for some  $x_{1}^{(i)},\cdots,x_{n_0}^{(i)}\in X$.
It follows from (\ref{651}) that
\begin{align}\label{652}
\mathcal{P}_{d}\big(\hat{f}_{0}^{m}(\mu_i),\hat{f}_{0}^{m}(\nu_i)\big)<\epsilon/2, \;1\leq m\leq b_k,\;1\leq i\leq k.
\end{align}
By the specification property of $(X,f_{0,\infty})$, for any $1\leq j\leq n_0$, there exists $z_j\in X$ such that
for any $1\leq i\leq k$, we have
\begin{align}\label{653}
d\big(f_{0}^{m}(z_j),f_{0}^{m}(x_j^{(i)})\big)\leq\epsilon/2,\;a_i\leq m\leq b_i.
\end{align}
Denote $\hat{\mu}=\frac{1}{n_0}\sum_{j=1}^{n_0}\delta_{z_j}$. By (\ref{653}) we have
\begin{align}\label{654}
\mathcal{P}_{d}\big(\hat{f}_{0}^{m}(\hat{\mu}),\hat{f}_{0}^{m}(\nu_i)\big)
=\mathcal{P}_{d}\left(\frac{1}{n_0}\sum_{j=1}^{n_0}\delta_{f_{0}^{m}(z_j)},\frac{1}{n_0}\sum_{j=1}^{n_0}\delta_{f_{0}^{m}(x_j^{(i)})}\right)
\leq\epsilon/2,\;a_i\leq m\leq b_i
\end{align}
for $1\leq i\leq k$.
It follows from (\ref{652}) and (\ref{654}) that for $1\leq i\leq k$,
\begin{align*}
\mathcal{P}_{d}(\hat{f}_{0}^{m}(\hat{\mu}),\hat{f}_{0}^{m}(\mu_i))\leq
\mathcal{P}_{d}(\hat{f}_{0}^{m}(\mu_i),\hat{f}_{0}^{m}(\nu_i))+\mathcal{P}_{d}(\hat{f}_{0}^{m}(\hat{\mu}),\hat{f}_{0}^{m}(\nu_i))
<\epsilon,\;a_i\leq m\leq b_i.
\end{align*}
Therefore, $(\mathcal{M}(X),\hat{f}_{0,\infty})$ has the specification property.

 (ii)
Fix any two nonempty open subsets $\mathcal{U}_{1},\mathcal{U}_{2}$ of $\mathcal{M}(X)$.
By Lemma \ref{3.2}, there exists $n_1\geq 1$ such that for any $1\leq j\leq 2$,
there exist $n_1$ nonempty open subsets $U_{1}^{j},\cdots,U_{n_1}^{j}$ of $X$ such that
$\frac{1}{n_1}\sum_{i=1}^{n_1}\delta_{y_i^{j}}\in \mathcal{U}_{j}$ for
$y_{i}^{j}\in U_{i}^{j}$ and $1\leq i\leq n_1$. Fix  $1\leq i\leq n_1$.
Since $(X,f_{0,\infty})$ has property $P$, there exists $N_i\in \mathbf{Z^{+}}$ satisfying the definition of property $P$
for $U_{i}^{1}$ and $U_{i}^{2}$. Fix any $k\geq2$ and any sequence $s=(s_1,s_2,\cdots,s_k)\in \Sigma_{2}^{k}$.
Denote
\begin{align*}
l_i=N_1\times\cdots \times N_{i-1}\times N_{i+1}\times\cdots \times N_{n_1},\;N=N_1\times\cdots\times N_{n_1},
\end{align*}
and
\begin{align*}
\hat{s}^{(i)}=(\underbrace{s_1,\cdots,s_1}_{l_i},\underbrace{s_2,\cdots,s_2}_{l_i},\cdots,\underbrace{s_k,\cdots,s_k}_{l_i})\in \Sigma_{2}^{kl_i}.
\end{align*}
Then there exists $z_i\in X$ such that
\begin{align*}
z_i\in U_{i}^{s_1}, f^{N}_{0}(z_i)\in U_{i}^{s_2},\cdots,f^{(k-1)N}_{0}(z_i)\in U_{i}^{s_k}.
\end{align*}
By defining  $\mu=\frac{1}{n_1}\sum_{i=1}^{n_1}\delta_{z_i}$, we have
\begin{align*}
\mu\in\mathcal{U}_{s_1},\hat{f}_{0}^{N}(\mu)=\frac{1}{n_1}\sum_{i=1}^{n_1}\delta_{f_{0}^{N}(z_i)}\in\mathcal{U}_{s_2},\cdots,
\hat{f}_{0}^{(k-1)N}(\mu)=\frac{1}{n_1}\sum_{i=1}^{n_1}\delta_{f_{0}^{(k-1)N}(z_i)}\in\mathcal{U}_{s_k}.
\end{align*}
Therefore, $(\mathcal{M}(X),\hat{f}_{0,\infty})$ has property $P$.
\end{proof}
The proof of Proposition \ref{specification property-property P} (i) is motivated by \cite{Bauer}. It is  interesting to prove or disprove the converse of Proposition \ref{specification property-property P}.
Finally, we  give  another difference between $(X,f)$ and $(\mathcal{M}(X),\hat f)$.

\begin{proposition}\label{shadowing-difference}
Assume that $f:X\to X$ is surjective. Then,

{\rm(i)} there exists  $(X,f)$ such that it has shadowing but does not  have average shadowing or specification property;

{\rm(ii)} shadowing of  $(\mathcal{M}(X),\hat f)$  implies average shadowing and specification property of $(\mathcal{M}(X), $ $\hat f)$.
\end{proposition}


\begin{proof}  (i) Consider again the example in Theorem \ref{discrete-example} (i.e., $X=\{a,b\}$ with the discrete metric and $f(a)=b$, $f(b)=a$).  Clearly, $(X,f)$ has shadowing. However, Example 3 in \cite{KO12} indicates that it  does not  have average shadowing. Theorem 1 in \cite{KO12} shows that it does not have specification property.

(ii)
By Theorem \ref{mixingshadowing}, $(X,f)$ is  mixing. Then Theorem \ref{mixing} implies that $(\mathcal{M}(X),\hat f)$ is  mixing.  Thus, it follows from Theorem 1 in \cite{KO12} that $(\mathcal{M}(X),\hat f)$ has average shadowing and specification property.
\end{proof}
\section{Sensitivity}
Let $A\subset\mathbf{N}$. $A$ is  cofinite if $\mathbf{N}\setminus A$ is finite;
$A=\{n_k\}_{k=1}^{\infty}$ is  syndetic if there exists $l\geq 1$ such that
$n_{k+1}-n_{k}\leq l$ for all $k\geq 1$; $A$ is  ergodic if
$\limsup_{n\to\infty}|A\cap\{0,1,\cdots,n-1\}|/n$ is positive, where $|A|$ denotes the cardinality
of $A$. $(X,f_{0,\infty})$ is sensitive if there exists $\delta>0$ such that $N_{d}(x,\epsilon,\delta)\neq\emptyset$
for any $x\in X$ and $\epsilon>0$, where
\begin{align*}
N_{d}(x,\epsilon,\delta)=\{n\geq 1:{\rm there\; exsits}\;
y\in B_{d}(x,\epsilon)\;{\rm such\; that}\;d\big(f_{0}^{n}(x),f_{0}^{n}(y)\big)>\delta\};
\end{align*}
it is cofinitely sensitive if  there exists $\delta>0$ such that $N_{d}(x,\epsilon,\delta)$ is cofinite for  $x\in X$ and $\epsilon>0$; it is syndetically sensitive if there exists $\delta>0$ such that
$N_{d}(x,\epsilon,\delta)$ is syndetic for $x\in X$ and $\epsilon>0$; it is ergodically sensitive if  there exists $\delta>0$ such that $N_{d}(x,\epsilon,\delta)$ is ergodic for $x\in X$ and $\epsilon>0$;
it is multi-sensitive if there exists $\delta>0$ such that  $\bigcap_{i=1}^{k}N_{d}(x_i,\epsilon,\delta)\neq\emptyset$ for $\epsilon>0$,  $k\geq 1$ and  $x_1,\cdots,x_k
\in X$; it is Li-Yorke sensitive if there exists $\delta>0$ such that for any $x\in X$ and $\epsilon>0$, there exists $y\in B_d(x,\epsilon)$ such that $(x,y)$ is a Li-Yorke $\delta$-pair (i.e.
$\liminf_{n\to\infty}d(f_{0}^{n}(x),f_{0}^{n}(y))=0$ and $\limsup_{n\to\infty}d(f_{0}^{n}(x), f_{0}^{n}(y))\geq\delta$).
Here, $\delta$ is a sensitivity constant.
Clearly, cofinite sensitivity implies multi-sensitivity and syndetic sensitivity, and syndetic sensitivity
implies ergodic sensitivity.

The following lemma will be needed.
\begin{lemma}\label{sen}
Let $x\in X$ and $0<\epsilon<\delta$. Then $N_{\mathcal{P}_{d}}(\delta_{x},\epsilon,\delta)
\subset N_{d}(x,\epsilon,\delta/2)$.
\end{lemma}

\begin{proof}
Let $n\in N_{\mathcal{P}_{d}}(\delta_{x},\epsilon,\delta)$.
Suppose that $n\notin N_{d}(x,\epsilon,\delta/2)$. Then for any $y\in B_{d}(x,\epsilon)$,
we have
$d(f_{0}^{n}(x),f_{0}^{n}(y))\leq\delta/2.$
This implies that
\begin{align}\label{b}
B_{d}(x,\epsilon)\subset f_{0}^{-n}\big(\bar{B}_{d}(f_{0}^{n}(x),\delta/2)\big)
\subset f_{0}^{-n}\big(B_{d}(f_{0}^{n}(x),\delta)\big).
\end{align}
Note that
\begin{align}\label{c}
1=\delta_{x}(\{x\})\leq\mu(B_{d}(x,\epsilon))+\epsilon<\mu(B_{d}(x,\epsilon))+\delta
\end{align}
for  any fixed $\mu\in B_{\mathcal{P}_{d}}(\delta_x,\epsilon)$.
Let $A\in \mathcal{B}(X)$. If $f_{0}^{n}(x)\notin A$, then
\begin{align}\label{p-c1}
0=\delta_{f_{0}^{n}(x)}(A)=\hat{f}_{0}^{n}(\delta_x)(A)\leq\hat{f}_{0}^{n}(\mu)(A^{\delta})+\delta.
\end{align}
If $f_{0}^{n}(x)\in A$, then it follows from  (\ref{b}) and (\ref{c}) that
\begin{align}\label{p-c2}
\hat{f}_{0}^{n}(\delta_x)(A)<\mu(B_{d}(x,\epsilon))+\delta
\leq\mu\big(f_{0}^{-n}(B_{d}(f_{0}^{n}(x),\delta))\big)+\delta\leq
\hat{f}_{0}^{n}(\mu)(A^{\delta})+\delta.
\end{align}
Combining (\ref{p-c1})--(\ref{p-c2}), we have $\mathcal{P}_{d}(\hat{f}_{0}^{n}(\delta_x),\hat{f}_{0}^{n}(\mu))\leq\delta$, which contradicts
the fact that $n\in N_{\mathcal{P}_{d}}(\delta_{x},\epsilon,\delta)$.
\end{proof}

Lemma \ref{sen} implies the following result.

\begin{proposition}\label{mx}
If $(\mathcal{M}(X),\hat f_{0,\infty})$ is sensitive {\rm(}reps., cofinitely sensitive, multi-sensitive,
syndetically  sensitive, and ergodically  sensitive{\rm)}, then so is $(X,f_{0,\infty})$.
\end{proposition}

The next two theorems show that cofinite sensitivity (resp., multi-sensitivity) of $(X,f_{0,\infty})$ is equivalent to that of  $(\mathcal{M}(X),\hat f_{0,\infty})$. Theorem \ref{cofinite} has interesting applications to interval autonomous systems, see Theorem \ref{sensitive equivalent interval}.

\begin{theorem}\label{cofinite}
$(X,f_{0,\infty})$ is cofinitely sensitive if and only if $(\mathcal{M}(X),\hat f_{0,\infty})$ is cofinitely sensitive.
\end{theorem}

\begin{proof}
It suffices to show the necessity by Proposition \ref{mx}.
Fix  $\mu\in \mathcal{M}(X)$ and $\epsilon>0$. Lemma \ref{basic result} (ii) implies that there exists
$\nu\triangleq\frac{1}{n_0}\sum\limits_{j=1}^{n_0}\delta_{x_{j}}\in B_{\mathcal{P}_{d}}(\mu,\epsilon/2)$
for some $n_0\in \mathbf{Z^{+}}$ and $x_{1},\cdots,x_{n_0}\in X$. Since $(X,f_{0,\infty})$
is cofinitely sensitive  with sensitivity constant $\delta>0$, there exists $N_0\geq 1$ such that
\begin{align}\label{411}
[N_0,+\infty)\subset\bigcap_{j=1}^{n_0}N_{d}(x_j,\epsilon/2,\delta).
\end{align}
Fix $k\geq N_0$. It follows from (\ref{411}) that for any $1\leq j\leq n_0$, there exists
$y_{j}\in B_{d}(x_{j},\epsilon/2)$ such that
\begin{align}\label{412}
d(f_{0}^{k}(x_{j}),f_{0}^{k}(y_{j}))>\delta.
\end{align}
Since $X$ is compact, there exists a finite open cover $\{V_m\}_{m=1}^{s}$ with
\begin{align}\label{413}
d(V_m)<\delta/2,\;\;1\leq m\leq s.
\end{align}
Let $\delta_0=\min\{\delta/8,1/s\}$. It can be shown that $(\mathcal{M}(X),\hat f_{0,\infty})$
is cofinitely sensitive with sensitivity constant $\delta_1<\delta_0/2$. Indeed, we have
\begin{align*}
1=\nu(X)=\nu\left(\bigcup_{m=1}^{s}f_{0}^{-k}(V_m)\right)\leq\sum_{m=1}^{s}\nu\big(f_{0}^{-k}(V_m)\big)=\sum_{m=1}^s\hat f_0^k(\nu)(V_m).
\end{align*}
Thus, there exists $1\leq m_0\leq s$ such that
\begin{align}\label{414}
\hat{f}_{0}^{k}(\nu)(V_{m_0})=\nu(f_0^{-k}(V_{m_0}))
=\frac{1}{n_0}\sum_{j=1}^{n_0}\delta_{x_{j}}\big(f_{0}^{-k}(V_{m_0})\big)
\geq\frac{1}{s}\geq\delta_0.
\end{align}
Then there exists $1\leq j_0\leq n_0$ such that $x_{j_0}\in f_{0}^{-k}(V_{m_0})\subset f_{0}^{-k}(V_{m_0}^{\delta_0})$,
which means that
\begin{align}\label{def-A}
A\triangleq\{1\leq j\leq n_0: x_{j}\in f_{0}^{-k}(V_{m_0}^{\delta_0})\}\neq\emptyset.
\end{align}
Fix $j\in A$. By (\ref{412})--(\ref{413}), we have for any $z\in V_{m_0}^{\delta_0}$,
\begin{align*}
d(f_{0}^{k}(y_{j}),z)\geq d(f_{0}^{k}(x_{j}),f_{0}^{k}(y_{j}))-d(f_{0}^{k}(x_{j}),z)
\geq\delta-d(V_{m_0}^{\delta_0})\geq\frac{\delta}{4},
\end{align*}
and thus
\begin{align}\label{415}
f_{0}^{k}(y_{j})\notin V_{m_0}^{\delta_0},\;\forall\;j\in A.
\end{align}
Define
\begin{align}\label{416}
\hat{\nu}\triangleq\frac{1}{|A|}\sum_{j\in A}\delta_{y_{j}}+\frac{1}{n_0-|A|}
\sum_{j\in\{1,\cdots,n_0\}\setminus A}\delta_{x_{j}}.
\end{align}
Let $B\in \mathcal{B}(X)$. If $x_{j}\in B$ for some $1\leq j\leq n_0$, then $d(y_{j},B)\leq d(y_{j}, x_{j})<\epsilon/2$
and thus $y_{j}\in B^{\epsilon/2}$. Consequently,
 \begin{align*}\nu(B)={1\over n_0}\sum_{j=1}^{n_0}\delta_{x_j}(B)\leq{1\over n_0}\sum_{j\in A}\delta_{y_j}(B^{\epsilon/2})+
 {1\over n_0}\sum_{j\in\{1,\cdots,n_0\}\setminus A}\delta_{x_j}(B^{\epsilon/2})\leq
 \hat{\nu}(B^{\epsilon/2})+\epsilon/2,\end{align*}
which implies  $\mathcal{P}_{d}(\nu,\hat{\nu})<\epsilon/2$.
Thus,
$
\mathcal{P}_{d}(\mu,\hat{\nu})\leq \mathcal{P}_{d}(\mu,\nu)+\mathcal{P}_{d}(\nu,\hat{\nu})<\epsilon.
$
Moreover, it follows from (\ref{def-A})--(\ref{415}) that $y_j\notin f_0^{-k}(V_{m_0}^{\delta_0})$ for $j\in A$,
and $x_j\notin f_0^{-k}(V_{m_0}^{\delta_0})$ for
$j\in\{1,\cdots,n_0\}\setminus A.$ This, along with (\ref{414}) and (\ref{416}), gives   $\hat{f}_{0}^{k}(\hat{\nu})(V_{m_0}^{\delta_0})=0$ and
\begin{align*}
\hat{f}_{0}^{k}(\hat{\nu})(V_{m_0}^{\delta_0})+\delta_0=\delta_0\leq\hat{f}_{0}^{k}(\nu)(V_{m_0}),
\end{align*}
which yields
\begin{align*}
\mathcal{P}_{d}(\hat{f}_{0}^{k}(\hat{\nu}),\hat{f}_{0}^{k}(\nu))\geq\delta_0.
\end{align*}
Thus, $\mathcal{P}_{d}(\hat{f}_{0}^{k}(\mu),\hat{f}_{0}^{k}(\nu))\geq\delta_0/2>\delta_1$ or
$\mathcal{P}_{d}(\hat{f}_{0}^{k}(\mu),\hat{f}_{0}^{k}(\hat{\nu}))\geq\delta_0/2>\delta_1$. Hence,
$
[N_0,+\infty)\subset N_{\mathcal{P}_{d}}(\mu,\epsilon,\delta_1).
$
Therefore, $(\mathcal{M}(X),\hat f_{0,\infty})$ is cofinitely sensitive.
\end{proof}

\begin{theorem}\label{multi sensitive}
$(X,f_{0,\infty})$ is multi-sensitive if and only if $(\mathcal{M}(X),\hat f_{0,\infty})$ is multi-sensitive.
\end{theorem}

\begin{proof}
Suppose that $(X,f_{0,\infty})$ is multi-sensitive.
 Fix $k\geq1$, $\mu_1,\cdots,\mu_k\in \mathcal{M}(X)$ and
$\epsilon>0$. By Lemma \ref{basic result} (ii),  there exist $\nu_{i}\triangleq\frac{1}{n_i}\sum_{j=1}^{n_i}
\delta_{x_{ij}}\in B_{\mathcal{P}_{d}}(\mu_i,\epsilon/2)$, $1\leq i\leq k$, for some $n_i\geq 1$ and $x_{i1},\cdots,x_{in_i}\in X$. Since $(X,f_{0,\infty})$ is multi-sensitive with sensitivity constant $\delta>0$, there exists $n_0\geq 1$
such that for any $1\leq i\leq k$ and $1\leq j\leq n_i$, there exists $\hat{x}_{ij}\in B_{d}(x_{ij},\epsilon/2)$
satisfying that
\begin{align}\label{46}
d(f_{0}^{n_0}(x_{ij}),f_{0}^{n_0}(\hat{x}_{ij}))>\delta.
\end{align}
Since $X$ is compact, there exists a finite open cover $\{V_m\}_{m=1}^{s}$ with
$
d(V_m)<\delta/2$ for all $1\leq m\leq s.
$
Denote $\delta_0=\min\{\delta/8,1/s\}$ and fix $1\leq i\leq k$. Since
$
1=\nu_{i}(X)\leq\sum_{m=1}^{s}\nu_i\big(f_{0}^{-n_{0}}(V_m)\big),
$
there exists $1\leq m_0\leq s$ such that
\begin{align*}
\hat{f}_{0}^{n_{0}}\big(\nu_{i}(V_{m_0})\big)=\frac{1}{n_i}\sum_{j=1}^{n_i}\delta_{x_{ij}}\big(f_{0}^{-n_{0}}(V_{m_0})\big)
\geq\frac{1}{s}\geq\delta_0.
\end{align*}
Thus, there exists $1\leq j\leq n_i$ such that $x_{ij}\in f_{0}^{-n_0}(V_{m_0}^{\delta_0})$. Then
$
A_i\triangleq\{1\leq j\leq n_i: x_{ij}\in f_{0}^{-n_{0}}(V_{m_0}^{\delta_0})\}\neq\emptyset.
$
Fix $j\in A_i$ and $z\in V_{m_0}^{\delta_0}$. By (\ref{46}) and the fact  that $d(V_{m_0})<{\delta/2}$, we have
\begin{align*}
d(f_{0}^{n_0}(\hat{x}_{ij}),z)\geq d(f_{0}^{n_0}(x_{ij}),f_{0}^{n_0}(\hat{x}_{ij}))-d(f_{0}^{n_0}(x_{ij}),z)
\geq\delta-d(V_{m_0}^{\delta_0})\geq\frac{\delta}{4},
\end{align*}
which means
$
f_{0}^{n_0}(\hat{x}_{ij})\notin V_{m_0}^{\delta_0}$ for all $j\in A_i.
$
Denfine
\begin{align*}
\hat{\nu}_{i}:=\frac{1}{|A_i|}\sum_{j\in A_i}\delta_{\hat{x}_{ij}}+\frac{1}{n_i-|A_i|}
\sum_{j\in \{1,\cdots,n_i\}\setminus A_i}\delta_{x_{ij}}.
\end{align*}
For any $E\in \mathcal{B}(X)$,  $x_{ij}\in E$ implies $\hat{x}_{ij}\in E^{\epsilon/2}$. Then $\nu_{i}(E)\leq\hat{\nu}_{i}(E^{\epsilon/2})+\epsilon/2$,
and thus $\mathcal{P}_{d}(\nu_{i},\hat{\nu}_{i})<\epsilon/2$. So,
$
\mathcal{P}_{d}(\mu_i,\hat{\nu}_{i})\leq \mathcal{P}_{d}(\mu_i,\nu_{i})+\mathcal{P}_{d}(\nu_i,\hat{\nu}_i)<\epsilon.
$
Since $\hat{f}_{0}^{n_0}(\hat{\nu}_i)(V_{m_0}^{\delta_0})=0$, we have
$
\hat{f}_{0}^{n_0}(\hat{\nu}_i)(V_{m_0}^{\delta_0})+\delta_0=\delta_0\leq\hat{f}_{0}^{n_0}(\nu_i)(V_{m_0}),
$
which yields that
$
\mathcal{P}_{d}(\hat{f}_{0}^{n_0}(\hat{\nu}_i),\hat{f}_{0}^{n_0}(\nu_i))\geq\delta_0.
$
Thus,
\begin{align*}
\mathcal{P}_{d}(\hat{f}_{0}^{n_0}(\mu_i),\hat{f}_{0}^{n_0}(\nu_i))\geq\delta_0/2\;\;{\rm or}\;\;
\mathcal{P}_{d}(\hat{f}_{0}^{n_0}(\mu_i),\hat{f}_{0}^{n_0}(\hat{\nu}_i))\geq\delta_0/2.
\end{align*}
Hence, $(\mathcal{M}(X),\hat f_{0,\infty})$ is multi-sensitive with sensitivity constant $\delta_1<\delta_0/2$.

The converse is proved  by Proposition \ref{mx}.
\end{proof}

For the general autonomous system  on   compact metric space,
there exists $(X,f)$ such that it is sensitive but $(\mathcal{M}(X),\hat f)$ is not sensitive (see Example 5.1 in
  \cite{Li17}). When considering the interval autonomous system, however,  we get the following equivalent result, which is very different from the above general case.

\begin{theorem}\label{sensitive equivalent interval}
{\rm(i)} $(I,f)$ is sensitive
if and only if $(\mathcal{M}(I),\hat{f})$ is sensitive.

{\rm(ii)} $(I,f)$ is syndetically sensitive
if and only if $(\mathcal{M}(I),\hat{f})$ is  syndetically sensitive.

{\rm(iii)} $(I,f)$ is  ergodically sensitive
if and only if $(\mathcal{M}(I),\hat{f})$ is  ergodically sensitive.
\end{theorem}

\begin{proof}
Necessity of (i)--(iii) is obtained by Proposition \ref{mx}. Now, we consider their sufficiency.

(i) Suppose that $(I,f)$ is sensitive. Then $(I,f)$ is cofinitely sensitive by Theorem 2 in \cite{Moothathu}. It follows from  Theorem \ref{cofinite} that
$(\mathcal{M}(I),\hat{f})$ is cofinitely sensitive, which implies that $(\mathcal{M}(I),\hat{f})$
is sensitive.

(ii)--(iii) Suppose that $(I,f)$ is syndetically  (resp., ergodically) sensitive. Then $(I,f)$ is sensitive. Again by Theorem 2 in \cite{Moothathu},
 $(I,f)$ is cofinitely sensitive. Theorem \ref{cofinite} shows that $(\mathcal{M}(I),\hat{f})$ is cofinitely sensitive. Thus, $(\mathcal{M}(I),\hat{f})$
is syndetically  (resp., ergodically) sensitive.
\end{proof}

It was proved in Theorem 4.1 of \cite{Li17} that Li-Yorke sensitivity of $(\mathcal{M}(X),\hat f)$ implies that of $(X,f)$. Now, it is generalized  to non-autonomous systems. It should be pointed out that there exists $(X,f)$ such that it is Li-Yorke sensitive but $(\mathcal{M}(X),\hat f)$ is not Li-Yorke sensitive (see
Theorem 5.2 in  \cite{Li17}).

\begin{proposition}\label{LiYorke}
If $(\mathcal{M}(X),\hat f_{0,\infty})$ is Li-Yorke sensitive, then so is $(X,f_{0,\infty})$.
\end{proposition}

\begin{proof}
Suppose that $(\mathcal{M}(X),\hat f_{0,\infty})$ is Li-Yorke sensitive with  sensitivity constant $2\delta>0$.
Let $x\in X$ and $0<\epsilon<\delta/2$. Then there exists $\mu\in B_{\mathcal{P}_{d}}(\delta_{x},\epsilon/2)$
such that $(\delta_{x},\mu)$ is a Li-Yorke $2\delta$-pair. Thus, there exist $\{n_i\}_{i=1}^{\infty}$ and
$\{m_i\}_{i=1}^{\infty}$ such that
\begin{align}\label{0415}
\mathcal{P}_{d}(\hat{f}_{0}^{n_i}(\delta_{x}),\hat{f}_{0}^{n_i}(\mu))<1/2^{i}\textrm{ and }\mathcal{P}_{d}(\hat{f}_{0}^{m_i}(\delta_{x}),\hat{f}_{0}^{m_i}(\mu))>\delta,\;i\geq1.
\end{align}
Then
\begin{align*}
1=\hat{f}_{0}^{n_i}(\delta_{x})(\{f_{0}^{n_i}(x)\})\leq\hat{f}_{0}^{n_i}(\mu)\left(B_d(f_{0}^{n_i}(x),1/2^{i})\right)
+1/2^{i},\;i\geq1,
\end{align*}
which yields that
\begin{align}\label{0417}
\mu\left(X\setminus f_{0}^{-n_{i}}\big(\bar{B}_d(f_{0}^{n_i}(x),1/2^{i})\big)\right)\leq1/2^{i},\;i\geq1.
\end{align}
Denote
\begin{align*}
D_{t}\triangleq\bigcap_{i=t}^{\infty}f_{0}^{-n_{i}}\left(\bar{B}_d\big(f_{0}^{n_i}(x),1/2^{i}\big)\right), \;t\geq1.
\end{align*}
Clearly, $D_t$ is closed for any $t\geq1$. By (\ref{0417}),
\begin{align*}
\mu(X\setminus D_t)\leq\sum_{i=t}^{+\infty}\mu\left(X\setminus f_{0}^{-n_{i}}\big(\bar{B}_d(f_{0}^{n_i}(x),1/2^i)\big)\right)
\leq\sum_{i=t}^{+\infty}1/2^{i}=1/2^{t-1}.
\end{align*}
So,
$
\mu(D_t)=1-\mu(X\setminus D_t)\geq1-1/2^{t-1}$ for all $t\geq1,
$
which implies that
\begin{align}\label{0419}
\lim_{t\to\infty}\mu(D_t)=1.
\end{align}
 By the second relation of (\ref{0415}), there exists $B_i\in\mathcal{B}(X)$ such that
$
\hat{f}_{0}^{m_i}(\delta_{x})(B_i)>\hat{f}_{0}^{m_i}(\mu)(B_i^{\delta})+\delta>0
$
for  $i\geq1$. Then $f_{0}^{m_i}(x)\in B_i$ and
$
1=\hat{f}_{0}^{m_i}(\delta_{x})(B_i)>\hat{f}_{0}^{m_i}(\mu)
(B_d(f_{0}^{m_{i}}(x),\delta))+\delta
$ for $i\geq1$.
Hence,
$
\mu\big(X\setminus f_{0}^{-m_{i}}(B_d(f_{0}^{m_{i}}(x),\delta))\big)>\delta$ for  $i\geq1.
$
Let
\begin{align*}
E\triangleq\bigcap_{t=1}^{\infty}\bigcup_{i=t}^{\infty}\left(X\setminus f_{0}^{-m_{i}}\big(B_d(f_{0}^{m_{i}}(x),\delta)\big)\right).
\end{align*}
Then
\begin{align*}
\mu(E)=\lim_{t\to\infty}\mu\left(\bigcup_{i=t}^{\infty}\left(X\setminus f_{0}^{-m_{i}}\big(B_d(f_{0}^{m_{i}}(x),\delta)\big)\right)\right)\geq\delta.
\end{align*}
By Theorem 6.1 in \cite{Walters}, there exists a closed subset $E_1\subset E$ such that
\begin{align}\label{0421}
\mu(E_1)\geq3\delta/4>3\epsilon/2.
\end{align}
Since $\mathcal{P}_{d}(\mu,\delta_{x})<\epsilon/2$, we have
$
1=\delta_{x}(\{x\})\leq\mu(B_d(x,\epsilon/2))+\epsilon/2.
$
Thus,
\begin{align}\label{0422}
\mu\big(B_d(x,\epsilon/2)\big)\geq1-\epsilon/2.
\end{align}
Denote
$
K_{t}\triangleq\bar{B}_d(x,\epsilon/2)\cap D_t\cap E_1$ for  $t\geq1.
$
Then, $K_{t}$ is closed and $K_{t}\subset B_d(x,\epsilon)$ for all $t\geq1$.
It follows from (\ref{0419})--(\ref{0422}) that there exists $N\geq 1$ such that $K_{t}\neq\emptyset$ for  $t\geq N$.
Let $y\in K_{t}$.  Then, $y\in D_t$ and thus
$
d(f_{0}^{n_i}(x),f_{0}^{n_i}(y))\leq1/2^{i}
$
for  $i\geq t\geq N$,
which proves that
\begin{align}\label{0423}
\liminf_{n\to\infty}d(f_{0}^{n}(x),f_{0}^{n}(y))=0.
\end{align}
On the other hand, $y\in K_t\subset E_1$  implies that there exists $i\geq t$ such that
$
d(f_{0}^{m_i}(x),f_{0}^{m_i}(y))\geq\delta
$
for $t\geq 1$.
Thus,
\begin{align}\label{0424}
\limsup_{n\to\infty}d(f_{0}^{n}(x),f_{0}^{n}(y))\geq\delta.
\end{align}
By (\ref{0423})--(\ref{0424}), $(x,y)$ is a Li-Yorke $\delta$-pair. Therefore,
$(X,f_{0,\infty})$ is Li-Yorke sensitive.
\end{proof}

\section{Topological sequence entropy, conjugacy and chaos}
Consider  non-autonomous system $(Y,g_{0,\infty})$, where $g_n: Y\to Y$ is a map on the compact metric space $(Y,\rho)$, $n\geq0$. $(X,f_{0,\infty})$ is said to be (topologically) $\{h_n\}_{n=0}^{\infty}$
-equi-semiconjugate to $(Y,g_{0,\infty})$ if there exists a sequence of equi-continuous and surjective maps $\{h_n\}_{n=0}^{\infty}$ from $X$ to $Y$
such that $h_{n+1}\circ f_n=g_n\circ h_n$ for all $n\geq0$.
$(X,f_{0,\infty})$ is   $\{h_n\}_{n=0}^{\infty}$
-equi-conjugate to $(Y,g_{0,\infty})$ if, in addition,
  $\{h_n^{-1}\}_{n=0}^{\infty}$ is  equi-continuous from $Y$ to $X$.
The  concept for  entropy is given in Subsection 2.1. A basic lemma related to  conjugacy and entropy is given as follows.

\begin{lemma}\label{entropy}
Let $(X,f_{0,\infty})$ be  $\{h_n\}_{n=0}^{\infty}$-equi-semiconjugate to $(Y,g_{0,\infty})$ and $A\subset\mathbf{Z^{+}}$ be  any increasing sequence.
Then $h_A(g_{0,\infty})\leq h_A(f_{0,\infty})$.
\end{lemma}

\begin{proof}
 Let $\epsilon>0$. Since $\{h_n\}_{n=0}^{\infty}$ is equi-continuous, there exists $0<\delta<\epsilon$
such that for any $x,y\in X$ and  $n\geq0$,
\begin{align}\label{51}
d(x,y)\leq\delta\Rightarrow\rho(h_{n}(x),h_{n}(y))<\epsilon.
\end{align}
Let $n\geq 1$ and $E_{Y}\subset Y$ be an $(n,\epsilon,A)$-separated set of $(Y,g_{0,\infty})$ with maximal cardinality $s_n(\epsilon,A,g_{0,\infty},Y)$. Since $h_0$ is surjective,  there exists $u_v\in X$ such that $h_{0}(u_v)=v$ for  $v\in E_{Y}$.
Let $E_{X}\triangleq\{u_{v}: v\in E_{Y}\}\subset X$.  Then there exists $0\leq j\leq n-1$ such that
$
\rho(g_{0}^{a_j}\circ h_{0}(u_{v_1}),g_{0}^{a_j}\circ h_{0}(u_{v_2}))=\rho(g_{0}^{a_j}(v_1),g_{0}^{a_j}(v_2))>\epsilon
$
for $v_1\neq v_2\in E_Y$.
By induction, we have
$
h_{k}\circ f_{0}^{k}=g_{0}^{k}\circ h_0 $ for any $k\geq1.
$
Hence,
$
\rho(h_{a_j}\circ f_{0}^{a_j}(u_{v_1}),h_{a_j}\circ f_{0}^{a_j}(u_{v_2}))>\epsilon.
$
This, along with (\ref{51}), implies that
$
d(f_{0}^{a_j}(u_{v_1}),f_{0}^{a_j}(u_{v_2}))>\delta.
$
Thus, $E_X$ is an $(n,\delta,A)$-separated set of $(X,f_{0,\infty})$. This proves that
$
s_n(\delta,A,f_{0,\infty},X)\geq s_n(\epsilon,A,g_{0,\infty},Y).
$
It follows from (\ref{topological sequence entropy}) that  $h_A(g_{0,\infty})\leq h_A(f_{0,\infty})$.
\end{proof}
Next, we compare $h_A(f_{0,\infty})$ and $h_A(\hat f_{0,\infty})$   for any increasing sequence $A\subset \mathbf{Z}^+$.
Recall that a nonempty subset $\Lambda$ of $X$ is  invariant with respect to $(X,f_{0,\infty})$ if $f_n(\Lambda)\subset \Lambda$ for all $n\geq0$, and
 $(\Lambda,f_{0,\infty})$  is an invariant subsystem of $(X,f_{0,\infty})$.

\begin{lemma}\label{sequence entropy}
Let  $A\subset\mathbf{Z^{+}}$ be  an increasing sequence. Then
$h_A(f_{0,\infty})\leq h_A(\hat{f}_{0,\infty})$.
In particular, $h(f_{0,\infty})\leq h(\hat{f}_{0,\infty})$.
\end{lemma}

\begin{proof}
Note that the map
\begin{align}\label{varphi1}\varphi_1: X\to\mathcal{M}_{1}(X)\end{align}
 defined in (\ref{2.3})
is a homeomorphism. In this way, $X$ is actually embedded in $\mathcal{M}(X)$. Moreover,
$
\varphi_1\circ f_{n}(x)=\delta_{f_{n}(x)}=\hat{f}_{n}(\delta_{x})=\hat{f}_{n}\circ\varphi_1(x)
$
for $x\in X$ and $n\geq0$. $(\mathcal{M}_{1}(X),\hat f_{0,\infty})$ is clearly an invariant subsystem of $(\mathcal{M}(X),\hat f_{0,\infty})$.
Thus, $(X,f_{0,\infty})$ and  $(\mathcal{M}_{1}(X),\hat f_{0,\infty})$   are  $\{\varphi_1\}$-equi-conjugate.
It follows from Lemma \ref{entropy} that $h_{A}(f_{0,\infty})=h_{A}\big(\hat{f}_{0,\infty},\mathcal{M}_{1}(X)\big)\leq h_A(\hat{f}_{0,\infty})$.
\end{proof}
Then, we show that  equi-continuity is interactively  preserved.
\begin{lemma}\label{equi-continuous}
Let $h_n: X\to Y$ be a map for $n\geq0$. Then, $\{h_n\}_{n=0}^{\infty}$ is equi-continuous  if and only if
$\{\hat{h}_n\}_{n=0}^{\infty}$ is equi-continuous.
\end{lemma}

\begin{proof}
Suppose that $\{h_n\}_{n=0}^{\infty}$ is equi-continuous. Let $\epsilon>0$.
Then there exists $\delta \in(0,\epsilon)$ such that for  $n\geq0$ and $x,y\in X$,
\begin{align}\label{5.1}
d(x,y)<\delta\Rightarrow\rho(h_{n}(x),h_{n}(y))<\epsilon.
\end{align}
We claim that
\begin{align}\label{5.2}
(h_{n}^{-1}(B))^{\delta}\subset h_{n}^{-1}(B^{\epsilon})
\end{align}
for $B\in\mathcal{B}(Y)$ and $n\geq0$.
In fact,  there exists $x'\in h_{n}^{-1}(B)$ such that
$d(x,x')<\delta$ for any $x\in(h_{n}^{-1}(B))^{\delta}$.  By (\ref{5.1}), $\rho(h_{n}(x),h_{n}(x'))<\epsilon$, and thus $x\in h_{n}^{-1}(B^{\epsilon})$.
For any $\mu,\nu\in \mathcal{M}(X)$ with  $\mathcal{P}_{d}(\mu,\nu)<\delta$, it follows from (\ref{5.2}) that
\begin{align*}
\hat{h}_{n}(\mu)(B)=\mu(h_n^{-1}(B))\leq\nu((h_{n}^{-1}(B))^{\delta})+\delta
\leq\nu(h_{n}^{-1}(B^{\epsilon}))+\epsilon=\hat{h}_{n}(\nu)(B^{\epsilon})+\epsilon,
\end{align*}
which means that $\mathcal{P}_{\rho}(\hat{h}_n(\mu),\hat{h}_n(\nu))<\epsilon$. Hence,
$\{\hat{h}_n\}_{n=0}^{\infty}$ is equi-continuous.

Conversely, suppose that $\{\hat{h}_n\}_{n=0}^{\infty}$ is equi-continuous. Let $\epsilon\in(0,1)$.
Then there exists $\delta_1\in(0,\epsilon)$ such that for any $n\geq0$ and $\mu,\nu\in\mathcal{M}(X)$,
$
\mathcal{P}_{d}(\mu,\nu)<\delta_{1}\Rightarrow\mathcal{P}_{\rho}(\hat{h}_n(\mu),\hat{h}_n(\nu))<\epsilon.
$
Let $n\geq0$ and  $x,y\in X$ with $d(x,y)<\delta_{1}$. Then, $\mathcal{P}_{d}(\delta_{x},\delta_{y})\leq d(x,y)<\delta_{1}$,
and thus $\rho(h_{n}(x),h_{n}(y))=\mathcal{P}_{\rho}(\hat{h}_n(\delta_{x}),\hat{h}_n(\delta_{y}))<\epsilon$. Therefore, $\{h_n\}_{n=0}^{\infty}$
is equi-continuous.
\end{proof}
Next, we prove the  preservation of   equi-conjugacy.

\begin{theorem}\label{equi conjugacy}
Let $h_n: X\to Y$ be a map  for  $n\geq0$. Then, $(X,f_{0,\infty})$ is
$\{h_n\}_{n=0}^{\infty}$-equi-conjugate to $(Y,g_{0,\infty})$ if and only if $(\mathcal{M}(X),\hat{f}_{0,\infty})$ is  $\{\hat{h}_n\}_{n=0}^{\infty}$-equi-conjugate to $(\mathcal{M}(Y),
\hat{g}_{0,\infty})$.
\end{theorem}

\begin{proof}
Suppose that $(X,f_{0,\infty})$ is  $\{h_n\}_{n=0}^{\infty}$-equi-conjugate to $(Y,g_{0,\infty})$. First, we show that $\hat h_n:\mathcal{M}(X)\to\mathcal{M}(Y)$ is bijective for $n\geq0$.
 Fix $n\geq0$. Let $\mu,\nu\in\mathcal{M}(X)$ with
$\hat{h}_n(\mu)=\hat{h}_n(\nu)$. For any $A\in\mathcal{B}(X)$, there exists $B_{0}\in\mathcal{B}(Y)$
such that $B_{0}=h_{n}(A)$ since $h_n$ is a homeomorphism. Thus,
$
\mu(A)=\mu(h_{n}^{-1}(B_0))=\hat{h}_n(\mu)(B_0)
=\hat{h}_n(\nu)(B_0)=\nu(h_{n}^{-1}(B_0))=\nu(A),
$
which means $\mu=\nu$, and thus $\hat{h}_n$ is injective.
Let $\nu\in\mathcal{M}(Y)$. Define
$\mu(A)\triangleq\nu(h_n(A))$ for  $A\in\mathcal{B}(X)$. Then $\mu\in\mathcal{M}(X)$.
For any $B\in\mathcal{B}(Y)$, there exists $A\in\mathcal{B}(X)$ such that $B=h_{n}(A)$, and thus,
$
\nu(B)=\nu(h_{n}(A))=\mu(A)=\mu(h_{n}^{-1}(B))=\hat{h}_n(\mu)(B).
$
Then $\nu=\hat{h}_n(\mu)$ and $\hat{h}_n$ is surjective.

Next, we show that $\{\hat{h}_n\}_{n=0}^{\infty}$ and $\{\hat{h}_{n}^{-1}\}_{n=0}^{\infty}$ are equi-continuous.
Note that $\hat{h}_n^{-1}=\widehat{h_n^{-1}}$. In fact, there exists $\mu\in\mathcal{M}(X)$ such that $\nu=\hat h_n(\mu)$
for $\nu\in \mathcal{M}(Y)$, and thus,
$
\widehat{h_n^{-1}}(\nu)(B)=\nu(h_n(B))=\hat h_n(\mu)(h_n(B))=\mu(B)=\hat h_n^{-1}(\nu)(B)
$
for  $B\in\mathcal{B}(Y)$.
Since $\{h_n\}_{n=0}^{\infty}$ and $\{h_{n}^{-1}\}_{n=0}^{\infty}$ are equi-continuous,
it follows from Lemma \ref{equi-continuous} that $\{\hat{h}_n\}_{n=0}^{\infty}$ and $\{\hat{h}_{n}^{-1}=\widehat{h_n^{-1}}\}_{n=0}^{\infty}$ are equi-continuous. Since $h_{n+1}\circ f_{n}=g_{n}\circ h_{n}$, we have
\begin{align*}
\hat{h}_{n+1}\circ\hat{f}_{n}(\mu)(B)=\mu((h_{n+1}\circ f_{n})^{-1}(B))=\mu((g_{n}\circ h_{n})^{-1}(B))=\hat{g}_{n}\circ \hat{h}_{n}(\mu)(B)
\end{align*}
for  $\mu\in\mathcal{M}(X)$ and  $B\in\mathcal{B}(Y)$,
which implies that $\hat{h}_{n+1}\circ\hat{f}_{n}=\hat{g}_{n}\circ \hat{h}_{n}$.
Hence, $(\mathcal{M}(X),\hat{f}_{0,\infty})$ is  $\{\hat{h}_n\}_{n=0}^{\infty}$-equi-conjugate to $(\mathcal{M}(Y),\hat{g}_{0,\infty})$.

Suppose that $(\mathcal{M}(X),\hat{f}_{0,\infty})$ is  $\{\hat{h}_n\}_{n=0}^{\infty}$-equi-conjugate to
$(\mathcal{M}(Y),\hat{g}_{0,\infty})$. Fix $n\geq0$. If $h_n(x)=h_n(y)$ for some  $x,y\in X$,
then
$\hat{h}_n(\delta_{x})=\delta_{h_{n}(x)}=\delta_{h_{n}(y)}=\hat{h}_n(\delta_{y})
$ and thus $\delta_{x}=\delta_{y}$, which means that $x=y$. Hence, $h_n$ is injective.
Since $\hat h_n$ is surjective,
there exists $\mu\in\mathcal{M}(X)$ such that $\hat{h}_n(\mu)=\delta_{y}$ for $y\in Y$.
Then $\mu(h_{n}^{-1}(\{y\}))=\hat{h}_n(\mu)(\{y\})=\delta_{y}(\{y\})=1$, and thus $h_{n}^{-1}(\{y\})\neq\emptyset$.
Hence, $h_n$ is surjective.
By Lemma \ref{equi-continuous} and the fact that $\hat{h}_n^{-1}=\widehat{h_n^{-1}}$,
 $\{h_n\}_{n=0}^{\infty}$ and $\{h_{n}^{-1}\}_{n=0}^{\infty}$ are equi-continuous.
Moreover,
$
\delta_{h_{n+1}\circ f_{n}(x)}=\hat{h}_{n+1}\circ\hat{f}_{n}(\delta_x)
=\hat{g}_{n}\circ\hat{h}_{n}(\delta_x)=\delta_{g_{n}\circ h_{n}(x)}
$
for $x\in X$,
which means that $h_{n+1}\circ f_{n}=g_{n}\circ h_{n}$.
Therefore, $(X,f_{0,\infty})$ is  $\{h_n\}_{n=0}^{\infty}$-equi-conjugate to $(Y,g_{0,\infty})$.
\end{proof}

Recall that $(X,f_{0,\infty})$ is Li-Yorke chaotic if it has an uncountable Li-Yorke scrambled set $S$; that is,
for any $x\neq y\in S$, $(x,y)$ is a Li-Yorke pair, namely,
\begin{align*}
\liminf_{n\to\infty}d(f_{0}^{n}(x),f_{0}^{n}(y))=0,\;\limsup_{n\to\infty}d(f_{0}^{n}(x), f_{0}^{n}(y))>0.
\end{align*}
Also, $(X,f_{0,\infty})$ is  distributionally chaotic if it has an uncountable distributional scrambled set $D$; that is,
for any $x\neq y\in D$,

(i) $\limsup_{n\to\infty}\frac{1}{n}\sum\limits_{i=1}^{n}\chi_{[0,\epsilon)}(d(f_{0}^{i}(x), f_{0}^{i}(y)))=1$ for any $\epsilon>0$,

(ii) $\liminf_{n\to\infty}\frac{1}{n}\sum\limits_{i=1}^{n}\chi_{[0,\delta)}(d(f_{0}^{i}(x), f_{0}^{i}(y)))=0$ for some $\delta>0$.

From the embedding of $X$ into $\mathcal{M}(X)$ by $\eqref{varphi1}$,
 it is easy to  prove that Li-Yorke  chaos (resp., distributional chaos) of $(X,f_{0,\infty})$ implies that of $(\mathcal{M}(X),\hat{f}_{0,\infty})$.

\begin{proposition}\label{Li-Yorke chaos1}
If $(X,f_{0,\infty})$ is Li-Yorke chaotic (resp., distributionally chaotic), then so is $(\mathcal{M}(X),\hat{f}_{0,\infty})$.
\end{proposition}

\begin{proof}
Suppose that  $S$ is an uncountable Li-Yorke scrambled set of $(X,f_{0,\infty})$. Denote
$\mathcal{S}\triangleq\{\delta_x: x\in S\}$. Then $\mathcal{S}\subset\mathcal{M}(X)$ is uncountable.
Fix  $\delta_x\neq\delta_y\in\mathcal{S}$. Then $x\neq y\in S$ and $(x,y)$  is a Li-Yorke pair. Thus,
\begin{align*}
\limsup_{n\to\infty}\mathcal{P}_{d}(\hat{f}^{n}_{0}(\delta_x),\hat{f}^{n}_{0}(\delta_y))
=\limsup_{n\to\infty}\mathcal{P}_{d}(\delta_{f^{n}_{0}(x)},\delta_{f^{n}_{0}(y)})
\geq\min\{\limsup_{n\to\infty}d(f^{n}_{0}(x),f^{n}_{0}(y)),1\}>0
\end{align*}
and
\begin{align*}
\liminf_{n\to\infty}\mathcal{P}_{d}(\hat{f}^{n}_{0}(\delta_x),\hat{f}^{n}_{0}(\delta_y))
=\liminf_{n\to\infty}\mathcal{P}_{d}(\delta_{f^{n}_{0}(x)},\delta_{f^{n}_{0}(y)})
\leq\liminf_{n\to\infty}d(f^{n}_{0}(x),f^{n}_{0}(y))=0.
\end{align*}
Hence, $\mathcal{S}$ is an uncountable Li-Yorke scrambled set of $(\mathcal{M}(X),\hat{f}_{0,\infty})$,
and thus $(\mathcal{M}(X),\hat{f}_{0,\infty})$ is Li-Yorke chaotic.
Distributional chaos of $(\mathcal{M}(X),\hat{f}_{0,\infty})$ can be  similarly proved.
\end{proof}

Consider the simple  autonomous dynamical system $(X,f)$, where
\begin{align}\label{one-point compactification of integers system}
X=\mathbf{Z}\cup\{\infty\},\quad
		f(n)=\left\{\begin{array}{ll}
		n+1,& n\in\mathbf{Z},\\
		\infty, & n=\infty,
		\end{array}\right.
	\end{align}
and $X$ is considered as  a one-point compactification of integers with  metric $d$.
In \cite{Guirao}, it was proved that $(X,f)$ has no Li-Yorke pair, but $(\mathcal{K}(X),\bar{f})$ is distributionally chaotic (and thus Li-Yorke chaotic). Next, we prove that $(\mathcal{M}(X),\hat{f})$ has no Li-Yorke pair for system \eqref{one-point compactification of integers system}.

\begin{example}\label{one-point compactification of integers no Li-Yorke}
Let  $(X,f)$ be given in  \eqref{one-point compactification of integers system}.
It is to prove  for any $\mu\in\mathcal{M}(X)$, we have
\begin{align}\label{converge in measure metric}
\mathcal{P}_d(\hat f^m(\mu),\delta_{\infty})\longrightarrow0
\end{align}
as $m\to \infty$. Assume that \eqref{converge in measure metric} is true. Then
$\delta_{\infty}$ is a fixed point of $\mathcal{M}(X)$  and  all other points in $\mathcal{M}(X)$ are asymptotic to $\delta_{\infty}$. Thus, $(\mathcal{M}(X),\hat{f})$ has no Li-Yorke pair.

To prove \eqref{converge in measure metric}, we first claim that  for any $\mu\in\mathcal{M}(X)$,
\begin{align}\label{mu-Express}
\mu=\sum_{n\in X}a_n\delta_{n},
\end{align}
where $a_n=\mu(n)\in[0,1]$ for $n\in X$. Since $\sum_{n\in X}a_n=1$ and
 \begin{align}\label{sum-an-deltan}
\sum_{n\in X}a_n\delta_{n}\left(\bigcup_{i=1}^\infty K_i\right)=
\sum_{n\in X}\sum_{i=1}^{\infty}a_n\delta_{n}\left( K_i\right)
=\sum_{i=1}^{\infty}\sum_{n\in X}a_n\delta_{n}\left( K_i\right)
 \end{align}
 for all countable collections   $\{K_{i}\}_{i=1}^{\infty}$ of pairwise disjoint  sets in $\mathcal{B}(X)$,
 we have $\sum_{n\in X}a_n\delta_{n}\in\mathcal{M}(X)$.
For any $A=\bigcup_{k}\{n_k\}\in\mathcal{B}(X)$ with $n_k\in X$, we have $\mu(A)=\mu(\bigcup_{k}\{n_k\})=\sum_k\mu(n_k)=\sum_ka_{n_k}$, and by \eqref{sum-an-deltan}, we have
$
\sum_{n\in X}a_n\delta_{n}(A)=\sum_ka_{n_k}=\mu(A)
$.
This proves \eqref{mu-Express}.

Now, we prove \eqref{converge in measure metric}. By \eqref{mu-Express},
\begin{align*}
&\hat{f}^m(\mu)(A)=\hat{f}^m\left(\sum_{n\in X}a_n\delta_{n}\right)(A)=\left(\sum_{n\in \mathbf{Z}}a_n\delta_{n}+a_{\infty}\delta_{\infty}\right)({f}^{-m}(A))\\
=&\sum_{n\in \mathbf{Z}}a_n\delta_{{f}^{m}(n)}(A)+a_{\infty}\delta_{f^m(\infty)}(A)
=\sum_{n\in \mathbf{Z}}a_n\delta_{m+n}(A)+a_{\infty}\delta_{\infty}(A)
\end{align*}
for any $ m\in\mathbf{N}$ and $A\in\mathcal{B}(X)$. This, along with \eqref{measure-metric-def}, implies that  \eqref{converge in measure metric} holds true if and only if for any fixed $\varepsilon>0$, there exists $N_0\in\mathbf{Z}^+$ such that
\begin{align}\label{converge in measure metric-equivalent condition}
\hat{f}^m(\mu)(A)=\sum_{n\in \mathbf{Z}}a_n\delta_{m+n}(A)+a_{\infty}\delta_{\infty}(A)\leq \delta_{\infty}(A^{\varepsilon})+\varepsilon
\end{align}
for any $m\in\mathbf{Z}\cap(N_0,\infty)$ and $A\in\mathcal{B}(X)$. Since $\sum_{n\in \mathbf{Z}}a_n\leq1$, we have for the fixed $\varepsilon>0$, there exists $ N_1\in\mathbf{Z}^+$ such that
\begin{align}\label{sum-decomposition1}
\left(\sum_{n\in\mathbf{Z}}-\sum_{n\in\mathbf{Z}\cap[-N_1,N_1]}\right)a_n\delta_{m+n}(A)
=\sum_{n\in\mathbf{Z}\setminus[-N_1,N_1]}a_n\delta_{m+n}(A)
\leq\sum_{n\in\mathbf{Z}\setminus[-N_1,N_1]}a_n<{\varepsilon\over2}
\end{align}
for any $m\in\mathbf{N}$ and $A\in\mathcal{B}(X)$. By Lemma $\ref{basic result}$ $(\mathrm{i})$, we have $\mathcal{P}_d(\delta_{m+n},\delta_{\infty})\leq d(m+n,\infty)\longrightarrow0$ as $m\to\infty$ for any fixed $n\in\mathbf{Z}$. Then there exists $N_0\in\mathbf{Z}^+$ such that $\mathcal{P}_d(\delta_{m+n},\delta_{\infty})\leq {\varepsilon\over N_1+1}$ for any $m\in\mathbf{Z}\cap(N_0,\infty)$, and thus,
\begin{align*}
\delta_{m+n}(A)\leq \delta_{\infty}(A^{\varepsilon\over N_1+1})+{\varepsilon\over N_1+1}
\end{align*}
for any $A\in\mathcal{B}(X)$, where $n\in\mathbf{Z}\cap[-N_1,N_1]$.
Then for any $m\in\mathbf{Z}\cap(N_0,\infty)$ and $A\in\mathcal{B}(X)$, we have
\begin{align}\label{sum-decomposition3}
\sum_{n\in\mathbf{Z}\cap[-N_1,N_1]}a_n\delta_{m+n}(A)\leq &\sum_{n\in\mathbf{Z}\cap[-N_1,N_1]}a_n\delta_{\infty}(A^{\varepsilon\over N_1+1})+\sum_{n\in\mathbf{Z}\cap[-N_1,N_1]}a_n{\varepsilon\over N_1+1}\\\nonumber
\leq&\sum_{n\in\mathbf{Z}\cap[-N_1,N_1]}a_n\delta_{\infty}(A^{\varepsilon})+{\varepsilon\over 2},
\end{align}
where we used the fact that $0\leq\sum_{n\in\mathbf{Z}\cap[-N_1,N_1]}a_n\leq1$.
 Then by \eqref{sum-decomposition1}--\eqref{sum-decomposition3}, we have
for any $m\in\mathbf{Z}\cap(N_0,\infty)$ and $A\in\mathcal{B}(X)$,
\begin{align*}\label{sum-decomposition}
&\sum_{n\in\mathbf{Z}}a_n\delta_{m+n}(A)-\sum_{n\in\mathbf{Z}}a_n\delta_{\infty}(A^{\varepsilon})
=\left(\sum_{n\in\mathbf{Z}}-\sum_{n\in\mathbf{Z}\cap[-N_1,N_1]}\right)a_n\delta_{m+n}(A)\\
+&
\sum_{n\in\mathbf{Z}\cap[-N_1,N_1]}a_n(\delta_{m+n}(A)-\delta_{\infty}(A^{\varepsilon}))
+\left(\sum_{n\in\mathbf{Z}\cap[-N_1,N_1]}-\sum_{n\in\mathbf{Z}}\right)a_n\delta_{\infty}(A^{\varepsilon})
<{\varepsilon\over2}+{\varepsilon\over2}+0=\varepsilon.
\end{align*}
Thus, for any $m\in\mathbf{Z}\cap(N_0,\infty)$ and $A\in\mathcal{B}(X)$,
\begin{align*}
\hat{f}^m(\mu)(A)=\sum_{n\in \mathbf{Z}}a_n\delta_{m+n}(A)+a_{\infty}\delta_{\infty}(A)\leq \sum_{n\in\mathbf{Z}}a_n\delta_{\infty}(A^{\varepsilon})+\varepsilon+a_{\infty}\delta_{\infty}(A^{\varepsilon})
=\delta_{\infty}(A^{\varepsilon})+\varepsilon.
\end{align*}
This proves \eqref{converge in measure metric-equivalent condition}.
\end{example}

\begin{remark} (1)
There is another example $(X,f)$ such that both $(X,f)$ and $(\mathcal{M}(X),\hat{f})$ have no Li-Yorke pair, but $(\mathcal{K}(X),\bar{f})$ is distributionally chaotic, where $f$ is the generic  homeomorphism of the Cantor space $X =\{0,1\}^{\mathbf{N}}$. For this example,
 Bernardes and Darji proved that $(X, f)$ has no Li-Yorke pair in Theorem $4.2$ of \cite{Bernardes-Darji2012}, while
Bernardes and Vermersch   confirmed that  $(\mathcal{K}(X), \bar{f})$ is
uniformly distributionally chaotic in Theorem $3.2$ of \cite{Bernardes-Vermersch2015} and then proved that $(\mathcal{M}(X), \hat{f})$ has no Li-Yorke pair in Theorem $1$ of \cite{Bernardes}.

(2)
Let $f$ be a homeomorphism of $X$ and  $(X,f)$ be not Li-Yorke chaotic. Then it follows from Corollary $2.4$ in \cite{BGKM02}  that $h(f)=0$. Then  $h(\hat f)=0$ due to  Theorem $\mathrm{A}$ in \cite{Glasner}. In the above examples, $(\mathcal{M}(X), \hat{f})$ is not Li-Yorke chaotic. It is interesting to study whether $(\mathcal{M}(X), \hat{f})$ is not Li-Yorke chaotic in general.
\end{remark}
\section*{Acknowledgement}
This research was partially supported by Hong Kong Research Grants Council (GRF Grant
CityU11200317),
 Postdoctoral Innovative Talents Support Program
(Grant No. BX20180151) and
China Postdoctoral Science Foundation
(Grant No. 2018M630266).

\end{CJK*}

\end{document}